%% file: motivic-tambara.tex
\documentclass{amsart}
\usepackage{verbatim}
\usepackage[textsize=scriptsize]{todonotes}

\usepackage{etoolbox}
\newtoggle{draft}

\togglefalse{draft}
\iftoggle{draft} {
\usepackage[margin=1.5in]{geometry}
}{ % else
\usepackage[margin=1in]{geometry}
}
\iftoggle{draft} {
\newcommand{\thought}[1]{\todo[color=gray!40]{#1}}
}{ % else
\newcommand{\thought}[1]{}
}

\input{preamble2.tex}

\newcommand{\E}{\mathrm{E}}
\def\SmQP{\mathrm{SmQP}{}}
\def\FEt{\mathrm{FEt}{}}
\def\fet{\mathrm{f\acute et}}

\def\inj{\mathrm{inj}}
\def\op{\mathrm{op}}
\def\all{\mathrm{all}}
\def\naive{\mathrm{naive}}
\def\EssSm{\Sm^{\text{ess}}}

\title{Motivic Tambara functors}
\date{\today}

\author{Tom Bachmann}
\address{Fakultät Mathematik, Universität Duisburg-Essen, Essen, Germany}
\email{\tomemail}

\begin{document}

\maketitle

\begin{abstract}
Let $k$ be a field and denote by $\SH(k)$ the motivic stable homotopy category.
Recall its full subcategory $\SH(k)^{\eff\heart}$ \cite{bachmann-very-effective}.
Write $\NAlg(\SH(k))$ for the category
of $\Sm$-normed spectra \cite{bachmann-norms}; recall that there is a forgetful
functor $U: \NAlg(\SH(k)) \to \SH(k)$. Let $\NAlg(\SH(k)^{\eff\heart}) \subset
\NAlg(\SH(k))$ denote the full subcategory on normed spectra $E$ such that $UE
\in \SH(k)^{\eff\heart}$. In this article we provide an explicit description of
$\NAlg(\SH(k)^{\eff\heart})$ as the category of effective
homotopy modules with étale norms, at least if $char(k) = 0$. A weaker statement
is available if $k$ is perfect of characteristic $> 2$.
\end{abstract}

\tableofcontents

\section{Introduction}

\subsection*{Norms and normed spectra}
In \cite{bachmann-norms}, we defined for every finite étale morphism $f: S' \to
S$ of schemes a symmetric monoidal functor of symmetric monoidal
$\infty$-categories $f_\otimes: \SH(S') \to \SH(S)$. If $S' = S^{\coprod n}$ and
$f$ is the fold map, then $f_\otimes: \SH(S') \wequi \SH(S)^n \to \SH(S)$ is the
$n$-fold smash product. These norm maps commute with arbitrary base change and
assemble into a functor
\[ \SH^\otimes: \Span(\Sch, \all, \fet) \to \widehat{\Cat}_\infty, (X
\xleftarrow{p} Y \xrightarrow{f} Z) \mapsto (\SH(X) \xrightarrow{f_\otimes p^*}
\SH(Z)). \]
Here $\Span(\Sch, \all, \fet)$ denotes the $(2,1)$-category of spans in schemes,
where the forward arrows are required to be finite étale \cite[Appendix
C]{bachmann-norms}. The category $\NAlg(\SH(S))$ is defined as the category of
sections of the restriction of $\SH^\otimes$ to $\Span(\Sm_S, \all, \fet)$,
cocartesian over backwards arrows \cite[Section 7]{bachmann-norms}. In other
words, an object $E \in \NAlg(\SH(S))$ consists of for each $X \in \Sm_S$ a
spectrum $E_X \in \SH(X)$, for each morphism $p: X \to Y \in \Sm_S$ an equivalence
$E_X \wequi p^* E_Y$, for each finite étale morphism $f: U \to V \in \Sm_S$ a
morphism $p_\otimes E_U \to E_V$, and infinitely many coherences among these
data. The forgetful functor $U: \NAlg(\SH(S)) \to \SH(S), E \mapsto E_S$ is
monadic \cite[Proposition 7.6(2)]{bachmann-norms} and in particular
conservative.

Consider the embedding $\phi: \Fin_* \wequi \Span(\Fin, \inj, \all)
\to \Span(\Fin, \all, \all) \to \Span(\Sm_S, \all, \fet)$, where $\inj$
denotes the class of injections and $\Fin \to \Sm_S$ is given by $X
\mapsto \coprod_X S$. This induces a functor $\NAlg(\SH(S)) \to \CAlg(\SH(S))$
which is conservative, monadic and comonadic \cite[Proposition
7.6(3)]{bachmann-norms}. We thus we view the category of normed spectra
$\NAlg(\SH(S))$ as an enhancement of the category $\CAlg(\SH(S))$ of
$E_\infty$-ring spectra.

\subsection*{Effective homotopy modules}

Recall the infinite suspension spectrum functor $\Sigma^\infty_+: \Sm_S \to
\SH(S)$. Let $\SH(S)^\veff \subset \SH(S)$ be the full subcategory
generated under colimits and
extensions by $\Sigma^\infty_+ \Sm_S$. We call this the category of very
effective spectra. Also denote by $\SH(S)^\eff \subset \SH(S)$ the localizing
subcategory generated by $\Sigma^\infty_+ \Sm_S$; this is the category of
effective spectra. By standard results, $\SH(S)^\veff \subset \SH(S)^\eff$ is the
non-negative part of a $t$-structure on $\SH(S)^\eff$
which is called the \emph{effective homotopy
$t$-structure}. We write $\SH(S)^{\eff\heart}$ for its heart and
$\tau_{\le 0}^\eff: \SH(S)^\veff \to \SH(S)^{\eff\heart}$ for the truncation
functor.

Note that if $E \in \SH(S)$ then we have the presheaf $\pi_0(E) \in Ab(\Sm_S), X
\mapsto [\Sigma^\infty_+ X, E]$. Moreover, if $f: X \to Y \in \Sm_S$ is finite
étale, then there is a canonical \emph{transfer map} $tr_f: \pi_0(E)(Y) \to
\pi_0(E)(X)$ \cite[Section 4]{bachmann-real-etale}.
In what follows, we will apply this in particular to $E \in \SH(S)^{\eff\heart}$.

\subsection*{Tambara functors}

We define $\NAlg(\SH(S)^{\eff\heart}) \subset \NAlg(\SH(S))$ to be
the full subcategory on those $E \in \NAlg(\SH(S))$ such that $UE \in
\SH(S)^{\eff\heart}$.
Now suppose that $E \in \NAlg(\SH(S)^{\eff\heart})$. Then the presheaf $\pi_0(E)
\in Ab(\Sm_S)$
acquires \emph{norm maps}. In other words, if $f: X \to Y \in \Sm_S$ is finite étale, then
there is an induced map $N_f: \pi_0(E)(X) \to \pi_0(E)(Y)$. This is a map of
\emph{sets}, not abelian groups; in other words it is not additive. Instead, the
maps $N_f$ satisfy a generalized distributivity condition related to the
transfers $tr_g$ \cite[Corollary 7.21]{bachmann-norms}, making $\pi_0(E)$ into a so-called
\emph{Tambara functor} \cite{tambara1993multiplicative} \cite[Definition 8]{bachmann-gwtimes}. Let us write $T(S)$
for the category of effective homotopy modules $E$ which are provided with norm maps
on $\pi_0(E) \in Ab(\Sm_S)$ in such a way that the distributivity condition is
fulfilled; see Definition \ref{def:tambara-1} below for details. Then we have a
factorisation
\[ \pi_0: \NAlg(\SH(S)^{\eff\heart}) \to T(S) \to Ab(\Sm_S). \]

\subsection*{Main results} For an
additive category $\scr C$ and $e \in \Z_{>0}$,
write $\scr C[1/e]$ for the full subcategory on
those objects $E \in \scr C$ such that $E \xrightarrow{e} E$ is an equivalence. For
example, $Ab(\Sm_S)[1/e]$ is the full subcategory of $Ab(\Sm_S)$ consisting of
those presheaves which are presheaves of $\Z[1/e]$-modules. Write
$\NAlg(\SH(S)^{\eff\heart})[1/e]$ for the full subcategory on those objects $E
\in \NAlg(\SH(S)^{\eff\heart})$ such that $UE \in \SH(S)^{\eff\heart}[1/e]$.
Similarly for $T(S)[1/e]$. With these
preliminaries out of the way, we can state our main results.

\begin{theorem*}[see Corollary \ref{corr:main}]
Let $k$ be a perfect field of exponential characteristic $e \ne 2$.
Then the restricted forgetful functor
\[ \NAlg(\SH(k)^{\eff\heart})[1/e] \to T(k)[1/e] \]
is an equivalence of categories.
\end{theorem*}

Along the way, we establish the following result of independent interest.
To state it, recall the category of generalized motivic complexes
$\widetilde{\DM}(k)$ and the functor $\widetilde M: \SH(k) \to \widetilde{\DM}(k)$
\cite{Deglise17}. Write $\widetilde{\DM}(k)^\eff$ for the localizing subcategory
generated by $\widetilde M(\SH(k)^\eff)$. Then $\widetilde{\DM}(k),
\widetilde{\DM}(k)^\eff$ afford $t$-structures whose definition is completely
analogous to those on $\SH(k), \SH(k)^\eff$. The functors $\widetilde M: \SH(k) \to
\widetilde{\DM}(k)$ and $\widetilde M^\eff: \SH(k)^\eff \to
\widetilde{\DM}(k)^\eff$ are right $t$-exact, and hence induce functors on the
hearts. We then have the following result.

\begin{theorem*}[see Theorem \ref{thm:HI0-gen-transfer}] Let $k$ be an infinite perfect field, $char(k)
\ne 2$. Then the
induced functors
\[ \widetilde M^\heart: \SH(k)^\heart \to \widetilde{\DM}(k)^\heart \]
and
\[ \widetilde M^{\eff\heart}: \SH(k)^{\eff\heart} \to \widetilde{\DM}(k)^{\eff\heart} \]
are equivalences of categories.
\end{theorem*}

The first half of this result was established by other methods in
\cite{ananyevskiy2017framed}.
The second half of this result is particularly noticeable. The reason is that
the heart $\widetilde{\DM}(k)^{\eff\heart}$ has a very explicit description: it
is equivalent to the category $\widetilde{\HI}(k)$ of homotopy invariant
Nisnevich sheaves with generalized transfers.\footnote{There is at the moment no
clear consensus on what to call this category. There are a priori different
definitions of ``generalized transfers'' \cite{morel-friedlander-milnor},
``Milnor--Witt transfers'' \cite{calmes2014finite}, ``equationally framed
transfers'' \cite{voevodsky2001notes}, and ``tangentially framed transfers''
\cite{EHKSY}, but they all coincide for (strictly) homotopy invariant sheaves of
abelian groups. In this article we stick to the first terminology.}

I am confident that this theorem also holds over finite fields, but this would
require certain results which are not yet in the literature.
In any case there are alternative descriptions of $\SH(k)^{\eff\heart}$ which
are good enough for us; see Remark \ref{rmk:infinite} for more details.
Putting these two theorems
together, we obtain the following.\footnote{To treat finite fields $k$
here, ``generalized transfers'' must be taken to mean framed transfers, since
the equivalence with Milnor--Witt transfers has not yet been proved in this case.}

\begin{corollary*}
The category $\NAlg(\SH(k)^{\eff\heart})[1/e]$ is equivalent to the category of
homotopy invariant Nisnevich sheaves with generalized transfers and finite
étale norms distributing over the finite étale transfers.
\end{corollary*}

\subsection*{Overview of the article}
In Section \ref{sec:background} we introduce some standing assumptions and
notation, beyond the notation already established in this introduction.

In Section \ref{sec:tambara-1} we introduce a first notion of motivic Tambara
functors, called motivic Tambara functors of the first kind.
These are effective homotopy modules $M \in \HI_0(k)$ together with for
each finite étale map $f: X \to Y \in \Sm_k$ a norm map $N_f: M(X) \to M(Y)$,
such that the norms distribute over the finite étale transfers in a suitable
fashion. In the remainder of this section we establish basic structural
properties of the category of motivic Tambara functors of the first kind.

In Section \ref{sec:tambara-2} we introduce a second notion of motivic Tambara
functors, called motivic Tambara functors of the second kind. These are
effective homotopy modules $M \in \HI_0(k)$ together with for every finite étale
morphism $f: X \to Y \in \Sm_k$ and every smooth and quasi-projective morphism
$p: W \to X$ a norm map $N_{f,W}: M(W) \to M(R_f(W))$, where $R_f(W) \in \Sm_Y$
denotes the \emph{Weil restriction} of $W$ along $f$. The norms are again
required to distribute over transfers in a suitable fashion. Note that if $p =
\id: X \to X$, then $R_f X \wequi X$ and so we obtain $N_{f,X}: M(X) \to M(Y)$.
In other words, any motivic Tambara functor of the second kind naturally induces
a motivic Tambara functor of the first kind. The main result of Section
\ref{sec:tambara-2} is that this is an equivalence of categories.

In Section \ref{sec:tambara-naive}, we introduce a third notion of motivic
Tambara functors, called naive motivic Tambara functors. These are just presheaves
of sets $M$ on $\FEt_S$ such that $M(X \coprod Y) \iso M(X) \times M(Y)$,
together with norm and transfer maps for
finite étale morphisms, satisfying a suitable distributivity condition. Here
$\FEt_S$ denotes the category of finite étale $S$-schemes. This definition is
closest to Tambara's original definition. Using one of Tambara's original
results, we
easily show that naive motivic Tambara functors are well-behaved under group
completion and localization. This is used at a key point in the proof of
the main result.

In Section \ref{sec:effective-homotopy-modules} we study in more detail the
category $\HI_0(k)$. Using abstract categorical arguments, we show that if
$char(k) \ne 2$ then $\HI_0(k) \wequi \widetilde{DM}^\eff(k)^\heart$. From this
we deduce that for $X \in \Sm_k$, the effective homotopy module $\E X :=
\ul{\pi}_0(\Sigma^\infty_+ X)_0 \in
\HI_0(k)$ is, in a suitable sense, generated under transfers and pullbacks
by the maps $Y \to X$ for $Y \in \Sm_k$.

In Section \ref{sec:normed-homotopy-modules-I} we introduce yet another notion
of motivic Tambara functors, called normed effective homotopy modules. This is
just the category $\NAlg(\SH(k)^{\eff\heart})$. We construct it more formally,
and establish some of its basic properties.

Finally in Section \ref{sec:normed-homotopy-modules-II} we put everything
together and prove the main theorem. To do so we first note that there is a
canonical functor $\rho: \NAlg(\SH(k)^{\eff\heart}) \to T^2(k)$, where $T^2(k)$
denotes the category of motivic Tambara functors of the second kind. This just
arises from the fact that, by construction, if $M \in
\NAlg(\SH(k)^{\eff\heart})$ then $M$ has certain norm maps, known to distribute
over transfers. Next, we observe that both 
$\NAlg(\SH(k)^{\eff\heart})$ and $T^2(k)$ admit monadic forgetful functors to the
category $\HI_0(k)$ of homotopy modules, and that $\rho$ is compatible with
these forgetful functors. It is thus enough to prove that the induced morphism
of monads is an isomorphism. This reduces to showing that if $X \in \Sm_k$ and
$M$ denotes the free normed effective homotopy module on $\E X$, then $M$ is
also the free motivic Tambara functor of the second kind on $\E X$. We do this
by noting that there is an explicit formula for $M$ as a large colimit, coming
from the identification of the free normed spectrum functor
\cite[Remark 16.25]{bachmann-norms}. From this we can verify the universal
property of $M$ as a motivic Tambara functor of the second kind by an
essentially elementary (but lengthy) computation.

\subsection*{Use of $\infty$-categories}
Throughout, we freely use the language of $\infty$-categories as set out in
\cite{lurie-htt, lurie-ha}. Unless explicitly mentioned otherwise, all
categories are $\infty$-categories, all colimits are homotopy colimits, and so
on. That being said, our main categories of interest are actually $1$-categories
(i.e. equivalent as $\infty$-categories to the nerve of an ordinary category).
In a $1$-category, the $\infty$-categorical notions of colimits etc. reduce to
their classical counterparts; so in many parts of this article the
traditional-sounding language indeed has the traditional meaning.

\subsection*{Acknowledgements}
I would like to thank Marc Hoyois for teaching me essentially everything I know
about $\infty$-categories, extensive discussions on normed spectra, and several
discussions regarding the results in this article. I would further like to thank
Maria Yakerson and an anonymous referee for comments on a draft of this article.

\section{Background and notation}
\label{sec:background}

Throughout, $k$ is a perfect field.

Recall that the objects $\Sigma^\infty_+ X \wedge \Gmp{n} \in \SH(k)$,
$X \in \Sm_k, n \in \Z$ generate the non-negative part of a $t$-structure, known as the
\emph{homotopy $t$-structure} \cite[Section 5.2]{morel-trieste}.
We write $\HI_*(k) \wequi \SH(k)^\heart$ for the category of homotopy modules
\cite[Theorem 5.2]{morel-trieste}.
The functor $i^\heart: \SH(k)^{\eff\heart} \to \SH(k)^\heart$ is fully faithful
\cite[Propositions 4 and 5]{bachmann-very-effective}. We write $\HI_0(k) \subset
\HI_*(k)$ for its essential image, and call it the category of \emph{effective
homotopy modules}.
If $X \in \Sm_k$ then $\Sigma^\infty_+ X \in \SH(k)^{\eff}_{\ge 0}$,
and we denote by $\E{X} \in \HI_0(k) \wequi \SH(k)^{\eff\heart}$ the
truncation. For $M \in \HI_0(k)$ and $X \in \Sm_k$ we abbreviate $\Hom(\E X, M)
=: M(X)$. The functor $\HI_0(k) \to Ab(\Sm_k), M \mapsto (X \mapsto M(X))$
factors through $Ab_{Nis}(\Sm_k)$ and the induced functor $\HI_0(k) \to
Ab_{Nis}(\Sm_k)$ is conservative and preserves limits and colimits (and also
$\HI_0(k)$ has all limits and colimits) \cite[Proposition
5(3)]{bachmann-very-effective}. Moreover, its image consists of unramified
sheaves \cite[Lemma 6.4.4]{morel2005stable}. Here $Ab(\Sm_k)$ denotes the
category of presheaves of abelian groups on $\Sm_k$, and $Ab_{Nis}(\Sm_k)$ the
category of Nisnevich sheaves of abelian groups.

Throughout we will be working with full subcategories $\scr C \subset \Sch_k$ which contain
$Spec(k)$ and are closed under finite étale extensions (and so in particular
finite coproducts). We also denote this condition by $\scr C \subset_\fet \Sch_k$.

Recall that if $f: X \to Y \in \Sm_k$ is a finite étale morphism, then the
functor $\SmQP_Y \to \SmQP_X, T \mapsto T \times_Y X$ has a right adjoint $R_f$
called \emph{Weil restriction} \cite[Theorem
7.6.4]{bosch2012neron}. Here $\SmQP_X$ denotes the category
of smooth and quasi-projective $X$-schemes. In particular if $Z \to X$ is finite
étale, then $Z \to X$ is smooth and affine, so smooth and quasi-projective, so
$R_f Z$ exists.

Recall that if $f: X \to Y \in \Sm_k$ is a finite étale morphism and $M \in
\HI_0(k)$, then there is a canonical \emph{transfer morphism} $tr_f: M(X) \to
M(Y)$. These transfer morphisms are natural in $M$ and $f$
\cite[Section 4]{bachmann-real-etale}.

\section{Motivic Tambara functors of the first kind}
\label{sec:tambara-1}

We now come to the most intuitive definition of a motivic Tambara functor as an
effective homotopy module with norms.

First recall the notion of an \emph{exponential diagram} \cite[Definition
7]{bachmann-gwtimes}: given finite étale morphisms $A \xrightarrow{q} X
\xrightarrow{f} Y$ in $\Sm_k$, the corresponding exponential diagram is
\begin{equation*}
\begin{CD}
X @<q<< A @<e<< X \times_Y R_f A \\
@VfVV   @.      @VpVV     \\
Y  @>{\wequi}>>   R_f X  @<{R_f(q)}<<   R_f A.
\end{CD}
\end{equation*}
Here $e$ is the $X$-morphism corresponding by adjunction to the identity $R_f A
\to R_f A$, and $p$ is the canonical projection.

\begin{definition} \label{def:tambara-1}
Let $\scr C \subset_\fet \Sm_k$.
A $\scr C$-Tambara functor of the first kind consists of an effective homotopy
module $M \in \HI_0(k)$, together with for each $f: X \to Y \in \scr C$ finite
étale a map \emph{of sets} $N_f: M(X) \to M(Y)$ such that:
\begin{enumerate}
\item For $X \in \scr C$ we have $N_{\id_X} = \id_{M(X)}$ and if $X
  \xrightarrow{f} Y \xrightarrow{g} Z$ are finite étale morphisms in $\scr C$,
  then $N_{gf} = N_g \circ N_f$.
\item Given a cartesian square
\begin{equation*}
\begin{CD}
X' @>f'>> Y'  \\
@Vp'VV  @VpVV \\
X  @>f>>   Y
\end{CD}
\end{equation*}
  with $X, Y \in \scr C$ and $p$ finite étale, the following diagram
  commutes
\begin{equation*}
\begin{CD}
M(X') @<f'^*<< M(Y')  \\
@VN_{p'}VV  @VN_pVV \\
M(X)  @<f^*<<   M(Y).
\end{CD}
\end{equation*}
\item Given finite étale morphisms $A \xrightarrow{q} X \xrightarrow{f} Y$ in
  $\scr C$, the following diagram (induced by the corresponding exponential
  diagram) commutes
\begin{equation*}
\begin{CD}
M(X) @<tr_q<< M(A) @>e^*>> M(X \times_Y R_f A) \\
@VN_fVV   @.              @VN_pVV     \\
M(Y)  @>{\wequi}>>   M(R_f X)  @<{tr_{R_f(q)}}<< M(R_f A).
\end{CD}
\end{equation*}
\end{enumerate}
A morphism $\phi: M_1 \to M_2$ of $\scr C$-Tambara functors of the first kind is a morphism of
the underlying effective homotopy modules such that for every $f: X \to Y$
finite étale, the following diagram commutes
\begin{equation*}
\begin{CD}
M_1(X) @>\phi(X)>> M_2(X) \\
@VN_fVV            @VN_fVV \\
M_2(Y) @>\phi(Y)>> M_2(Y).
\end{CD}
\end{equation*}
We denote the category of $\scr C$-Tambara functors of the first kind by
$T_\scr{C}^1(k)$, and we write $U_1: T_\scr{C}^1(k) \to \HI_0(k)$ for the evident
forgetful functor.
\end{definition}

\begin{remark}\label{rmk:norm-inv}
If $f: X \to Y \in \scr C$ is an isomorphism, then condition (2) with $Y'=X'=X$,
$p=f$, $p'=f'=\id$ implies that
$N_f = (f^*)^{-1}$. It follows thus from condition (1) that for $f: X \to Y \in
\scr C$ finite étale, the map $N_f: M(X) \to M(Y)$ is invariant under
automorphisms of $X/Y$.
\end{remark}

\begin{remark}
If $M \in T_\scr{C}^1(k)$ and $X \in \scr C$, then the fold map $\nabla: X \coprod X \to X$
induces a binary operation $N_\nabla: M(X \coprod X) \wequi M(X) \times M(X) \to M(X)$ (the first isomorphism
because $M$ is a sheaf), called multiplication. This operation is commutative by Remark
\ref{rmk:norm-inv}. If $f: \emptyset \to X$ is the unique map\thought{should say somewhere that
this counts as finite étale...}, then $M(\emptyset) = *$ (since $M$ is a sheaf)
and $N_f(*) \in M(X)$ is a unit of this multiplication on $M(X)$ (this follows
from condition (1)).
Condition (3) implies that multiplication distributes over addition in
$M(X)$ (apply it to the sequence of étale maps $X \coprod (X \coprod X)
\xrightarrow{\id \coprod \nabla} X \coprod X \xrightarrow{\nabla} X$).
For this reason we refer to condition (3) as the \emph{distributivity
law}. We thus see that $M(X)$ is naturally a commutative ring.

Condition (2) implies that $M(X \coprod X) \wequi M(X) \times M(X)$ \emph{as
rings}, and condition (1) then implies
that for $f: X \to Y$ finite étale, the map $N_f: M(X) \to M(Y)$ is
multiplicative.
\end{remark}

\begin{remark}
If $\scr C = \FEt_k$, then the above definition coincides with \cite[Definition
8]{bachmann-gwtimes}.
\end{remark}

Here is a basic structural property of the category of $\scr C$-Tambara functors
of the first kind.

\begin{lemma} \label{lemm:T1-pres}
The category $T_\scr{C}^1(k)$ is presentable and the forgetful functor $U_1:
T_\scr{C}^1(k) \to \HI_0(k)$ is a right adjoint.
\end{lemma}
\begin{proof}\thought{Is there an abstract nonsense proof?}
We first construct auxiliary categories $\scr D$ and $\scr D'$. The objects
of both $\scr D$ and $\scr D'$ are
objects of $\scr C$. For $X, Y \in \scr C$, the morphisms from $X \to Y$ in
$\scr D'$ are given by equivalence classes \emph{spans}, i.e. diagrams $X
\xleftarrow{f} T \to Y$, where $f$ is required to be finite étale.
In other words $\scr D'$ is just the homotopy $1$-category of the bicategory
$\Span(\scr C, \fet, \all)$.

The morphisms from $X \to Y$ in
$\scr D'$ are given by equivalence classes
of \emph{bispans}, i.e.
diagrams $X \xleftarrow{f} T_1 \xleftarrow{g} T_2 \xrightarrow{p} Y$, where $f$
and $g$ are required to be finite étale. We shall identify two bispans if they
fit into a commutative diagram
\begin{equation*}
\begin{CD}
X @<f_1<< T_1 @<g_1<< T_2 @>p_1>> Y \\
@|        @VaVV       @VbVV      @| \\
X @<f_2<< T_1' @<g_2<< T_2' @>p_2>> Y
\end{CD}
\end{equation*}
with $a, b$ isomorphisms. If $f: X \to Y \in \scr C$ then we denote the bispan
$X \xleftarrow{\id} X \xleftarrow{\id} X \xrightarrow{f} Y$ by $\rho_f$. If $f:
X \to Y \in \scr C$ is finite étale, we denote the bispan $X \xleftarrow{f} Y
\xleftarrow{\id} Y \xrightarrow{\id} Y$ by $\tau_f$ and we denote the bispan $X
\xleftarrow{\id} X \xleftarrow{f} Y \xrightarrow{\id} Y$ by $\nu_f$.

Before explaining composition in $\scr D$, let us explain what the category
$\scr D$ is supposed to do. We will have a functor $F: \scr D' \to \scr D$ which is
the identity on objects and sends $X \xleftarrow{g} T \xrightarrow{f} Y$ to
$\rho_f \circ \tau_g$. This induces $F^*:
PSh(\scr D) \to PSh(\scr D')$. The objects in $PSh(\scr D)$ are going to be
``presheaves with norm and transfer'' in the following sense. Let $G: \HI_0(k)
\to PSh(\scr D)$ denote the forgetful functor. Then we have a
cartesian square of $1$-categories
\begin{equation} \label{eq:tambara-cart-square}
\begin{CD}
PSh(\scr D) @<<< T_\scr{C}^1(k) \\
@VF^*VV           @VU_1VV       \\
PSh(\scr D') @<G<< \HI_0(k).
\end{CD}
\end{equation}
The category $\HI_0(k)$ is presentable, being an accessible localization of the
presentable category $\SH(k)^\veff$. The categories $PSh(\scr D)$ and $PSh(\scr
D')$ are of course presentable. The functor $F^*$ has a left adjoint, given by
left Kan extension. The functor $G$ also has a left adjoint; indeed it is a
functor between presentable categories which preserves limits and filtered
colimits (see Lemma \ref{lemm:evX-sifted-colim} below), so the claim follows
from the adjoint functor theorem \cite[Corollary Corollary
5.5.2.9(2)]{lurie-htt}.
It follows that
$F^*$ and $G$ are morphisms in $Pr^R$. Thus the square is also a pullback in
$Pr^R$ \cite[Theorem 5.5.3.18]{lurie-htt}, and in particular $T_\scr{C}^1(k)$ is
presentable and $U_1$ is a right adjoint.

It remains to finish the construction of $\scr D$. The composition in $\scr D$
is determined by the following properties: (1) if $\alpha = (X \xleftarrow{f}
T_1 \xleftarrow{g} T_2 \xrightarrow{p} Y)$ is a bispan, then $\alpha = \rho_p
\nu_g \tau_f$. (2) if $X \xrightarrow{f} Y \xrightarrow{g} Z \in \scr C$, then
$\rho_{gf} = \rho_g \rho_f$. If $f, g$ are finite étale then $\tau_{gf} =
\tau_f \tau_g$ and $\nu_{gf} = \nu_f \nu_g$. (3) The $\tau$ and $\nu$ morphisms
satisfy the basechange law with respect to the $\rho$ morphisms. (4) the
distributivity law holds. For a more detailed construction of similar categories,
see \cite[Section 5, p. 24 and Proposition 6.1]{strickland2012tambara}.
\end{proof}

\begin{remark}
The cartesian square \eqref{eq:tambara-cart-square} can be used to elucidate the
nature of motivic Tambara functors of the first kind: the category is equivalent
to the category of triples $(T, M, \alpha)$ where $T$ is a presheaf on a certain bispan
category $\scr D$, $M$ is an effective homotopy module, and
$\alpha$ is an isomorphism between the presheaves with finite étale transfers underlying $T$ and
$M$. In fact, if $char(k) = 0$ then one may show that the functor
$\HI_0(k) \to PSh(\scr D')$ is fully faithful (use \cite[Corollary
5.17]{bachmann-moving} and \cite[paragraph before Proposition
22]{bachmann-gwtimes}), whence so is $T_\scr{C}^1(k) \to PSh(\scr D)$.
We deduce that in this situation the category $T_\scr{C}^1(k)$ has a
particularly simple description: it consists of presheaves on the bispan
category $\scr D$ such that the underlying presheaf with finite étale transfers
extends to an effective homotopy module (in particular, is a strictly homotopy
invariant sheaf).
\end{remark}

We immediately deduce the following.

\begin{corollary} \label{corr:T-1-limits-colimits}
The category $T_\scr{C}^1(k)$ has all (small) limits and colimits.
\end{corollary}

Recall now that if $F$ is a presheaf on a category $\scr D$, then $F$ extends
uniquely to a continuous presheaf on $Pro(\scr D)$, the category of pro-objects. Moreover,
consider the subcategory $\EssSm_k \subset \Sch_k$ on those schemes which can
be obtained as cofiltered limits of smooth $k$-schemes along diagrams with
affine transition morphisms. Then $\EssSm_k$ \emph{embeds} into $Pro(\Sm_k)$
\cite[Proposition 8.13.5]{EGAIV}, and consequently for $X \in \EssSm_k$ the expression $F(X)$
makes unambigious sense, functorially in $X$\thought{perhaps this should be
earlier}.  
It follows in particular that Definition \ref{def:tambara-1} makes sense more
generally for $\scr C \subset_\fet \EssSm_k$.

Let $\scr C \subset \Sch$. Write $\scr C^{sl}$ for the subcategory of $\Sch$ on
those schemes obtained as semilocalizations of schemes in $\scr C$ at finitely
many points. We write $\scr C \subset_{\fet,\op} \Sch$ to mean
that $\scr C \subset_\fet \Sch$ and $\scr C$ is closed under passage to open
subschemes.

A convenient property of the category $T_\scr{C}^1(k)$ is that, in reasonable
cases, it is invariant under replacing $\scr C$ by $\scr C^{sl}$:

\begin{proposition} \label{prop:T1-sl-equiv}
Let $\scr C \subset_{\fet,\op} \Sm_k$. Then
$\scr C^{sl} \subset_\fet \EssSm_k$ and the canonical forgetful functor $T_{\scr
C}^1(k) \to T_{\scr C^{sl}}^1(k)$ is an equivalence of categories.
\end{proposition}
In the proof, we shall make use of the unramifiedness property of homotopy
modules \cite[Lemma 6.4.4]{morel2005stable}: if $X \in \Sm_k$
is connected and $\emptyset \ne U \subset X$, then $M(X) \to M(U)$ is injective.
In particular, if $\eta$ is the generic point of $X$, then $M(X) \hookrightarrow
M(\eta)$.
\begin{proof}
Let $X \in \scr C^{sl}$ and $f: Y \to X$ finite étale. Then $X$ is a cofiltered
limit along open immersions, so there exists a cartesian square
\begin{equation*}
\begin{CD}
Y  @>i'>> V \\
@VfVV  @Vf'VV \\
X  @>i>> U \\
\end{CD}
\end{equation*}
with $U \in \scr C$, $i$ an open immersion and $f'$ finite étale \cite[Théorèmes
8.8.2(ii) and 8.10.5(x), and Proposition 17.7.8(ii)]{EGAIV}. It follows
that $V \in \scr C$, and $Y$ is a cofiltered limit (intersection) of open
subschemes of $V$. Since $Y \to X$ is finite (so in particular closed and
quasi-finite \cite[Tags 01WM and 02NU]{stacks-project}), $Y$ is semilocal, and so must be a
semilocalization of $V$. This proves the first claim.

Note that $T^1_{\scr C}(k) \to T^1_{\scr C^{sl}}(k)$ is full. Indeed if $M_1,
M_2 \in T^1_{\scr C}(k)$ and $\alpha: U_1 M_1 \to U_1 M_2$ is a morphism of the
underlying homotopy modules, compatible with the norms on semilocal schemes,
then it is compatible with the norms on generic points, and hence it is
compatible with all norms, by Lemma \ref{lemm:T1-gen-full} below.

The functor $T^1_{\scr C}(k) \to T^1_{\scr C^{sl}}(k)$
is also faithful, since $U_1: T^1_\scr{C}(k) \to
\HI_0(k)$ and $U_1^{sl}: T^1_{\scr C^{sl}}(k) \to \HI_0(k)$ are. It remains to
show that it is essentially surjective.

Thus let $M \in T^1_{\scr C^{sl}}(k)$. Let $p: X \to Y \in \scr C$ be finite
étale. We need to construct a norm $N_p: M(X) \to M(Y)$. We may assume that $Y$
is connected. Let $\eta$ be the generic point of $Y$. We are given a norm map
$N_p^\eta: M(X_\eta) \to M(\eta)$. By unramifiedness (and compatibility of norms
with base change), there is at most one map
$N_p$ compatible with $N_p^\eta$. What we need to show is that $N_p^\eta(M(X))
\subset M(Y)$. In order to do this, by unramifiedness again, it suffices to
prove this for $Y$ replaced by the various localizations of $Y$ at its points of
codimension one. But then $Y$ is semilocal, so the result holds by assumption.

It remains to show that these norms turn $M$ into a $\scr C$-motivic Tambara
functor of the first kind. The base change formula (i.e. condition (2) of
Definition \ref{def:tambara-1}) is satisfied by construction. It implies using
unramifiedness of $M$ that it is enough to check conditions (1) and (3) when the
base is a field (use that Weil restriction commutes with arbitrary base change
\cite[Proposition A.5.2(1)]{conrad2015pseudo}), in which case they hold by assumption.
\end{proof}

\begin{remark} \label{rmk:T1-C-comp}
It follows that if $\scr C_1, C_2 \subset_{\fet,\op} \Sm_k$ such that $\scr
C_1^{sl} = \scr C_2^{sl}$, then $T^1_{\scr C_1}(k) \wequi T^1_{\scr C_1}(k)$.
This applies for example if $\scr C_1 = \Sm_k$ and $\scr C_2 = \SmQP_k$.
\end{remark}

In the course of the proof of Proposition \ref{prop:T1-sl-equiv}, we used the
following lemma of independent interest.  
Let $\scr C \subset \Sch$. Write $\scr C^{gen}$ for the full subcategory of
$\Sch$ consisting of the subschemes of generic points of schemes in $\scr C$.
\begin{lemma} \label{lemm:T1-gen-full}
Let $\scr C \subset_{\fet,\op} \Sm_k$. Then $\scr C^{gen} \subset_\fet
\EssSm_k$.

Let $F,
G \in T^1_\scr{C}(k)$ and let $\alpha \in \Hom_{\HI_0(k)}(F,G)$ be a morphism of
the underlying homotopy modules. If $\alpha \in
\Hom_{T^1_{\scr{C}^{gen}}(k)}(F,G)$, then $\alpha \in
\Hom_{T^1_{\scr{C}}(k)}(F,G)$. If $U_1G \in \HI_0(k)[1/e]$ where $e$ is the
exponential characteristic of $k$, then the above criterion we may replace $\scr
C^{gen}$ by $\scr C^{gen,perf}$, consisting of the perfect closures of objects
in $\scr C^{gen}$.
\end{lemma}
\begin{proof}
The first claim is proved exactly as in the proof of Lemma
\ref{prop:T1-sl-equiv}. Suppose given $\alpha$ with the claimed properties. We
need to show that, if $p: X \to Y \in \scr C$ is finite étale, then the
following diagram commutes
\begin{equation*}
\begin{CD}
F(X) @>\alpha>> G(X) \\
@VN_pVV         @VN_pVV \\
F(Y) @>\alpha>> G(Y).
\end{CD}
\end{equation*}
Let $Y^{(0)}$ denote the set of generic points of $Y$. Then $G(Y) \to
G(Y^{(0)})$ is injective, by unramifiedness.
The base change formula thus allows us to assume that $Y \in \scr C^{gen}$. It
follows that $X \in \scr C^{gen}$, and so the diagram commutes by definition.

For the last claim, we use that if $X \in \scr C^{gen}$ has perfect closure
$X'$, then $G(X) \to G(X')$ is injective \cite[Lemma 17]{bachmann-hurewicz}.
\end{proof}

Corollary \ref{corr:T-1-limits-colimits} assures us that $T^1_\scr{C}(k)$ has
all limits and colimits. In the final part of this section,
we wish to investigate how these limits and
colimits are computed. We begin with the case of homotopy modules:

\begin{lemma} \label{lemm:evX-sifted-colim}
Let $X \in \EssSm_k$. The functor $ev_X: \HI_0(k) \to Ab, F \mapsto F(X)$ preserves
finite limits and filtered colimits.
If $X \in \Sm_k^{sl}$, then the functor $ev_X$ preserves arbitrary
colimits as well, whereas if $X \in \Sm_k$, it preserves arbirary limits.
\end{lemma}
\begin{proof}
By \cite[Proposition 5(3)]{bachmann-very-effective}, the functor $o: \HI_0(k) \to
Shv_{Nis}(\Sm_k)$ preserves limits and colimits. Taking global sections of
sheaves preserves limits, so the claim about preservation of limits when $X \in
\Sm_k$ is clear. If $X \in \EssSm_k$, say $X = \lim_i X_i$ (the limit being
cofiltered), then $F(X) = \colim_i F(X_i)$. Since this colimit is filtered, and
finite limits commute with filtered colimits, the claim about preservation of
finite limits follows.

A filtered colimit
of Nisnevich sheaves, computed in the category $PSh(\Sm_k)$, is still a
Nisnevich sheaf, since Nisnevich sheaves are detected by the distinguished
squares and filtered colimits commute with finite limits of sets. It follows
that $ev_X$ preserves filtered colimits (for any $X$). Since $ev_X$ preserves
finite limits and our categories are abelian, $ev_X$ preserves finite sums. Now let $X$ be
semi-local. It remains to show that $ev_X$ preserves cokernels. Let $\alpha: F \to G \in
\HI_0(k)$, let $K = ker(\alpha), C = cok(\alpha), I = im(\alpha)$ and
consider the short exact sequences $0 \to K \to F \to I \to 0$ and $0 \to I
\to G \to C \to 0$. It is enough to show that $ev_X$ preserves these exact
sequences. Let $0 \to F \to G \to H \to 0 \in \HI_0(k)$ be an exact sequence.
Then $0 \to o(F) \to o(G) \to o(H) \to 0$ is exact, and hence to show
that $0 \to F(X) \to G(X) \to H(X) \to 0$ is exact it suffices to show that
$H^1_{Nis}(X, o(F)) = 0$. This is proved in \cite[last paragraph of Theorem
10.12]{bachmann-norms} (if $k$ is infinite, this follows directly from
\cite[Lemma 3.6]{bachmann-criterion}, noting that $o(F)$ has $MW$-transfers, e.g. by
Theorem \ref{thm:HI0-gen-transfer}).
\end{proof}

We can deduce the desired result.

\begin{corollary} \label{corr:T1-sifted-colimits}
Let $\scr C \subset_{\fet,\op} \Sm_k$. Then $U_1: T^1_\scr{C}(k) \to \HI_0(k)$
preserves sifted colimits.
\end{corollary}
Note that a functor between categories with small colimits (such as ours)
preserves sifted colimits if and only if it preserves filtered colimits and
geometric realizations \cite[Corollary 5.5.8.17]{lurie-htt}, which for
1-categories (such as ours) is the same as preserving filtered colimits and
reflexive coequalizers. We will not use this observation.
\begin{proof}
By Lemma \ref{prop:T1-sl-equiv}, we may replace $\scr C$ by $\scr C^{sl}$. Let
$F: \scr D \to T^1_{\scr C^{sl}}(k)$ be a sifted diagram, and let $C =
\colim_\scr{D} U_1F$. Note that the forgetful functor $Ab \to Set$ preserves
sifted colimits. Hence if $X \in \scr C^{sl}$ then by Lemma
\ref{lemm:evX-sifted-colim} we find that $C(X) = \colim_\scr{D} F(\bullet)(X)$,
where the colimit is taken in the category \emph{of sets}. In particular if $f:
X \to Y \in \scr C^{sl}$ is finite étale, then there is a canonical induced norm
$N_f: C(X) \to C(Y)$. It is easy to check that $C$, together with these norms,
defines an object of $T^1_{\scr C^{sl}}(k)$ which is a colimit of $F$. This
concludes the proof.
\end{proof}

\section{Motivic Tambara functors of the second kind}
\label{sec:tambara-2}

We now come to a second, somewhat more technical definition of a category of
motivic Tambara functors. We will eventually show that in good cases,
it coincides with the first definition.

\begin{remark}
For the purposes of this article, Tambara functors of the second kind can be
viewed just as a technical tool: the proof of our main theorem (that normed
effective
homotopy modules are the same as Tambara functors of the first kind) is just
naturally split into showing both that normed effective homotopy modules are the same as
Tambara functors of the second kind, and that Tambara functors of the first and
second kind are the same.

Slightly more philosophically, it seems that Tambara functors of the second kind
are closer to the ``true'' nature of normed effective homotopy modules (in cases where
Tambara functors of the first and second kind are not the same); see also Remark
\ref{rmk:1-2}.
\end{remark}

\begin{definition}
If $\scr V \subset Mor(\Sm_k)$ is a class of smooth morphisms, then we call
$\scr V$ admissible if it contains the finite étale morphisms and is closed under
composition, base change, and Weil restriction along finite étale morphisms
(i.e. if $f: X \to Y \in \scr V$ and $p: Y \to Z$ is finite étale, then $R_p(X)$
exists and $R_p(f): R_p(X) \to R_p(Y) \wequi Z \in \scr V$).

We call $\scr V$ and $\scr C \subset_\fet \Sm_k$ \emph{compatible} if for all $f: X \to Y \in \scr C$
finite étale and all $V \to X \in \scr V$ we have $R_f V \in \scr C$.
\end{definition}

\begin{remark}
%Since any morphism of quasi-projective $k$-schemes is quasi-projective \cite[Tag
%0C4N]{stacks-project}, it follows that
The class of smooth quasi-projective morphisms is admissible.
Since finite étale schemes are smooth
quasi-projective, and Weil restriction preserves finite étale schemes
(this follows for example from \cite[Proposition 7.5.5]{bosch2012neron}),
we deduce that the class of finite étale morphisms is also
admissible.
\end{remark}

\begin{example}
Note that the following pairs $(\scr C, \scr V)$ are compatible: $(\Sm_k,
\SmQP)$, $(\SmQP_k, \SmQP)$, $(\FEt_k, \fet)$. However, $(\FEt_k, \SmQP)$ is
\emph{not} compatible.
\end{example}

\begin{definition} \label{defn:tambara-2}
Let $\scr C \subset_\fet \Sm_k$ and $\scr V \subset Mor(\Sm_k)$ admissible.
A $(\scr C,\scr V)$-Tambara functor of the second kind consists of an effective homotopy
module $M \in \HI_0(k)$, together with for each $f: X \to Y \in \scr C$ finite
étale and $V \to X \in \scr V$ a map \emph{of sets} $N_{f,V}: M(V) \to
M(R_f V)$ such that the following hold.
\begin{enumerate}
\item For $X \in \scr C$ and $V \to X \in \scr V$, we have $N_{\id_X, V} =
  \id_{M(V)}$. Moreover given $X \xrightarrow{f} Y \xrightarrow{g} Z \in \scr C$
  with $f, g$ finite étale we have $N_{gf, V} = N_{g, R_f V} \circ N_{f, V}$,
  under the canonical isomorphism $R_{gf} V \wequi R_g R_f V$.

\item Consider a commutative diagram
\begin{equation*}
\begin{CD}
V_1      @>a>>  V_2     \\
@Vp_1VV         @Vp_2VV \\
X'       @>p'>> Y'      \\
@Vf'VV          @VfVV   \\
X        @>p>>  Y
\end{CD}
\end{equation*}
in $\scr \Sm_k$. Suppose that the lower square is cartesian, $f$ is finite étale,
$X, Y \in \scr C$
and $p_1, p_2 \in \scr V$. Then the following diagram commutes
\begin{equation*}
\begin{CD}
M(V_1)               @<a^*<< M(V_2)                \\
@V{N_{f',V_1}}VV       @V{N_{f,V_2}}VV \\
M(R_{f'} V_1)        @<c^*<< M(R_{f} V_2).
\end{CD}
\end{equation*}
Here $c: R_{f'} V_1 \to R_f V_2$ corresponds by adjunction to a map $R_{f'} V_1
\to p^* R_f V_2 \wequi R_{f'} p'^* V_2$, which comes via $R_{f'}$ from a map $V_1
\to p'^* V_2$, namely the one which corresponds by adjunction to $a: V_1 \to
V_2$.

\item Consider a commutative diagram
\begin{equation*}
\begin{CD}
V_1   @>p_1>> X @>f>> Y \\
@VgVV        @|         \\
V_2   @>p_2>> X
\end{CD}
\end{equation*}
in $\scr \Sm_k$, with $f, g$ finite étale, $X, Y \in \scr C$
and $p_1, p_2 \in \scr V$. Then the
following diagram commutes
\begin{equation*}
\begin{CD}
M(V_1) @>{N_{f,V_1}}>> M(R_f V_1)        \\
@V{tr_g}VV                   @V{tr_{R_f(g)}}VV \\
M(V_2) @>{N_{f,V_2}}>> M(R_f V_2).
\end{CD}
\end{equation*}
\end{enumerate}

A morphism $\phi: M_1 \to M_2$ of $(\scr C, \scr V)$-Tambara functors of the second kind is a morphism of
the underlying effective homotopy modules such that for each $f: X \to Y \in
\scr C$ finite étale and $V \to X \in \scr V$ the following diagram commutes
\begin{equation*}
\begin{CD}
M_1(V) @>\phi(V)>>         M_2(V) \\
@V{N_{f,V}}VV              @V{N_{f,V}}VV     \\
M_1(R_f V) @>\phi(R_f V)>> M_2(R_f V).
\end{CD}
\end{equation*}

We denote the category of $(\scr C, \scr V)$-Tambara functors of the second kind by
$T_\scr{C}^2(k)$, and we write $U_2: T_\scr{C}^2(k) \to \HI_0(k)$ for the evident
forgetful functor. Observe that we suppress $\scr V$ from the notation.
\end{definition}

We give a special name to some of the simplest norm maps on a motivic Tambara
functor of the second kind.

\begin{construction} \label{constr:norms}
Let $M \in T_\scr{C}^2(k)$, and $f: X \to Y \in \scr C$ be finite étale. We define $N_f^{(1)}: M(X)
  \to M(Y)$ as $N_f^{(1)} = N_{f, X}$. Note that this makes sense, since $\id:
X \to X \in \scr V$ by assumption.
\end{construction}

In some sense, these special norms already determine all the norms:

\begin{lemma} \label{lemm:extra-norms}
Let $\scr C \subset_\fet \Sm_k$, $\scr V \subset Mor(\Sm_k)$ admissible, and
assume that $\scr C, \scr V$ are compatible.

Let $M \in T_\scr{C}^2(k)$
and $V \xrightarrow{p} X \xrightarrow{f} Y \in \scr C$ with $p \in \scr V$
and $f$ finite étale. Let $f': X \times_Y R_f V \to R_f V$ be the projection and
$a:X \times_Y R_f V \to V$ the counit. Then $N_{f,V} = N_{f'}^{(1)} a^*: M(V)
\to M(R_f V)$.
\end{lemma}
\begin{proof}
Apply Definition \ref{defn:tambara-2}(2) to the diagram
\begin{equation*}
\begin{CD}
X \times_Y R_f V @>a>> V     \\
@|                     @VpVV \\
X \times_Y R_f V @>>>  X     \\
@Vf'VV                 @VfVV \\
R_f V            @>R_f(p)>>  Y,
\end{CD}
\end{equation*}
noting that $c=\id$. Note that this makes sense: we have $R_f V \in \scr C$ by the compatibility assumption,
and then $X \times_Y R_f V \in \scr C$ since $\scr C$ is closed under finite étale
extension.
\end{proof}

We are now ready to prove our main result of this section.

\begin{proposition} \label{prop:comparison}
Let $\scr C \subset_\fet \Sm_k$, $\scr V \subset Mor(\Sm_k)$ admissible, and
assume that $\scr C, \scr V$ are compatible.

If $M \in T_\scr{C}^2(k)$ then $U_2 M \in \HI_0(k)$ together with the norm maps
$N_f^{(1)}$ of Construction \ref{constr:norms} defines a $\scr C$-motivic Tambara
functor of the first kind. Moreover the induced functor $T_\scr{C}^2(k) \to
T_\scr{C}^1(k)$ is an equivalence.
\end{proposition}
\begin{proof}
Write $F: T_\scr{C}^2(k) \to T_\scr{C}^1(k)$ for this (so far hypothetical)
functor. We begin by constructing what will be its inverse $G: T_\scr{C}^1(k) \to
T_\scr{C}^2(k)$. Thus let $M \in T_\scr{C}^1(k)$, $f: X \to Y \in \scr
C$ finite étale and $V \to X \in \scr V$. Consider the span $R_f(V)
\xleftarrow{f'} X \times_Y R_f(V) \xrightarrow{a} V$,
where $f'$ is the projection and $a$ is the counit. Then we put $N_{f,V} =
N_{f'} a^*: M(V) \to M(R_f(V))$. Write $GM$ for $M$ equipped with these norm
maps $N_{f,V}$; this is enough data to define an object of $T_\scr{C}^2(k)$ (but
we have not shown that the required conditions hold).

I claim that (a) if $M_2 \in
T_\scr{C}^2(k)$ then $FM_2$ is indeed a $\scr C$-motivic Tambara functor of the first kind,
and that (b) if $M_1 \in T_\scr{C}^1(k)$ then $GM_1$ is indeed a $(\scr C, \scr
V)$-motivic Tambara functor
of the second kind. Suppose for now that this is true. It is then clear that $F,
G$ are functors, i.e. send morphisms to morpisms. It follows from Lemma
\ref{lemm:extra-norms} that $GFM_2 = M_2$. Moreover $FGM_1 = M_1$ by
construction. Hence $F$ is an equivalence as claimed.
It thus remains to establish (a) and (b).

\textbf{Proof of (a).}
Conditions (1) and
(2) of Definition \ref{def:tambara-1} follow respectively from Definition
\ref{defn:tambara-2}(1) (with $V=X$) and (2) (with $p_1 = \id$ and $p_2 = \id$).
For condition (3), we use that $N_p^{(1)} e^* = N_{f,A}$ by Lemma
\ref{lemm:extra-norms}, and hence the condition follows from Definition
\ref{defn:tambara-2}(3) (with $V_1 = A, V_2 = X$). Note that $A \to X \in \scr
V$ by assumption, since $A \to X$ is finite étale.

\textbf{Proof of (b).}
Let $M \in T_\scr{C}^1(k)$.
We need to show that Conditions \ref{defn:tambara-2}(1--3) hold.

\textbf{Proof of (1).}
The condition about identities is clear. For the composition, let $X
\xrightarrow{f} Y \xrightarrow{g} Z \in \scr C$ be finite étale morphisms
and $V \to Y \in \scr V$. Consider the commutative diagram
\begin{equation*}
\begin{CD}
V @=  V @=  V  \\
@AaAA @AaAA @VVV \\
X \times_Z R_{gf} V @>>> X \times_Y R_f V @>>> X \\
@VpVV                      @VpVV               @VfVV \\
Y \times_Z R_{gf} V @>a>> R_f V            @>>> Y \\
@VpVV                      @.                 @VgVV \\
R_{gf} V  @>>>            Z     @=             Z.
\end{CD}
\end{equation*}
The maps $a$ are counit maps (use $R_{gf} V \wequi R_g R_f V$),
the maps $p$ are projections (use $X \times_Z
R_{gf} \wequi X \times_Y (Y \times_Z R_{gf} V)$). By definition $N_{gf,V}$ is
the composite $N_p N_p a^*$ induced by the left column, whereas $N_{g,R_f
V} N_{f, V}$ is the composite $N_p a^* N_p a^*$ induced by (first row, middle)
to (second row, middle) to (third row, middle) to (third row, left) to (fourth row, left).
The condition follows from Definition \ref{def:tambara-1}(2), because the middle left square is cartesian.

\textbf{Proof of (2).}
Consider the diagram
\begin{equation} \label{eq:tm-diag-comm}
\begin{CD}
M(V_2)  @>d^*>> M(Y' \times_Y R_f V_2) @>N>> M(R_f V_2) \\
@Va^*VV               @Vg^*VV                   @Vc^*VV          \\
M(V_1)  @>e^*>> M(X' \times_X R_{f'} V_1) @>N>> M(R_{f'} V_1).
\end{CD}
\end{equation}
Here $d, e$ are counit maps, $N$ means norm along the canonical projections,
and $g: X' \times_X R_{f'} V_1 \to Y' \times_Y R_f V_2)$ corresponds by
adjunction to a morphism $X' \times_X R_{f'} V_1 \iso f'^* R_{f'} V_1 \to p'^*
f^* R_f V_2 \iso f'^* p^* R_{f'} V_1$, namely the morphism $f'^* c$.
We note that the following square is cartesian
\begin{equation*}
\begin{CD}
X' \times_X R_{f' V_1} @>g>> Y' \times_Y R_f V_2 \\
@VVV                                @VVV        \\
R_{f' V_1}             @>c>> R_f V_2,
\end{CD}
\end{equation*}
where the vertical morphisms are the canonical projections.
It follows from Definition \ref{def:tambara-1}(2) that
the right hand square in diagram \eqref{eq:tm-diag-comm} commutes.
Moreover the following square commutes
\begin{equation*}
\begin{CD}
X' \times_X R_{f' V_1} @>g>> Y' \times_Y R_f V_2 \\
@VeVV                         @VdVV               \\
V_1                    @>a>>  V_2,
\end{CD}
\end{equation*}
and hence the left hand square in diagram  \eqref{eq:tm-diag-comm}
commutes. It follows that the outer rectangle also commutes, which is what we
needed to show.

\textbf{Proof of (3).}
Consider the following commutative diagram in which all
rectangles are cartesian
\begin{equation} \label{eq:diag-temp-2}
\begin{CD}
                    @.  V_1 @>g>> V_2 \\
@.                      @AAA     @AaAA \\
X \times_Y R_f V_1 @>q>> C   @>>> X \times_Y R_f V_2 @>>> X \\
@Vr'VV                    @.       @VrVV                  @VfVV \\
R_f V_1         @=  R_f V_1 @>>> R_f V_2            @>>> Y.
\end{CD}
\end{equation}
Here $r, r'$ are the canonical projections, $a$ is the counit map, and $q$ is
induced by the universal property of $C$ from the counit map $X \times_Y R_f V_1
\to V_1$ and $X \times_Y R_f(g): X \times_Y R_f V_1 \to X \times_Y R_f V_2$.
I claim that the canonically induced map $R_r C \to R_f V_1$ is an isomorphism.
In order to do this, let $T \in \Sch_{R_f V_2}$. We compute
\[ [T, R_f C]_{R_f V_2} \iso [r^* T, C]_{X\times_Y R_f V_2} \iso [r^* T, V_1]_{V_2}. \]
Here $[\dots]_X$ denotes the morphisms of $X$-schemes, and $r^*: \Sch_{R_f V_2}
\to \Sch_{X \times_Y R_f V_2}$ is the base change functor. The first isomorphism is
by definition (of Weil restriction),
and the second is because the top square in diagram \eqref{eq:diag-temp-2} is
cartesian. Here we view $r^* T$ as
a scheme over $V_2$ via $a$. We also have
\[ [T, R_f V_1]_{R_f V_2} \iso [f^* T, V_1]_{V_2}, \]
where on the right hand side we view $T$
as a scheme over $Y$ via the canonical map $R_f V_2 \to Y$. It
remains to observe that $f^* T = r^* T$, because the lower right hand square in
diagram \eqref{eq:diag-temp-2} is cartesian. With this preparation out of the
way, consider the diagram
\begin{equation*}
\begin{CD}
M(R_r C) @<N<< M(r^* R_r C) @<<< M(C) @<<< M(V_1) \\
@V{tr}VV        @.               @VtrVV    @VtrVV \\
M(R_f V_2) @<N<< M(X \times_Y R_f V_2) @= M(X \times_Y R_f V_2) @<<< M(V_2).
\end{CD}
\end{equation*}
The unlabelled arrows are restriction along some canonical map, the arrows
labelled $tr$ are transfer along some canonical map, and the arrows labelled $N$
are norm along some canonical map. The right hand square commutes by the base
change formula, and the left hand rectangle commutes by the distributivity law
(i.e. Definition \ref{def:tambara-1}(3)).
The top composite is $N_{f,V_1}$ (using that $R_r C
\iso R_f V_1$, as established above, and $r^* R_r C \iso f^* R_f V_1$, by
transitivity of base change) and the bottom composite is $N_{f,V_2}$,
so commutativity is precisely condition (3). This concludes the proof.
\end{proof}

\section{Naive Motivic Tambara functors}
\label{sec:tambara-naive}

Throughout this section, $k$ is an arbitrary base scheme. In particular it is
not necessarily a field, unless otherwise specified.

Recall the following very naive definition of a motivic Tambara functor
\cite[Definition 8]{bachmann-gwtimes}. It is closest to Tambara's original definition.

\begin{definition}
Let $k$ be some base scheme. A \emph{naive motivic Tambara functor} over
$k$
is a presheaf of sets $M$ on $\FEt_k$ which preserves finite products (when viewed as
a functor $\FEt_k^\op \to \mathrm{Set}$), provided with for every (necessarily
finite étale) morphism $f: X \to Y \in \FEt_k$ two further maps of sets $tr_f,
N_f: M(X) \to M(Y)$, such that the following conditions hold:
\begin{enumerate}
\item $tr_{\id_X} = N_{\id_X} = \id_{M(X)}$ and if $X \xrightarrow{f} Y
  \xrightarrow{g} Z \in \FEt_k$, then $tr_{gf} = tr_g \circ tr_f$ and $N_{gf} =
  N_g \circ N_f$.
\item $tr_f$ and $N_f$ satisfy the base change formula. (For $N_f$ this is the
  condition (2) of Definition \ref{def:tambara-1}, with $\scr C = \FEt_k$. For
  $tr_f$ the condition is the same, just with $tr$ in place of $N$ everywhere.)
\item The distributivity law holds (in the sense of Definition
  \ref{def:tambara-1}(3)).
\end{enumerate}
The morphisms of naive motivic Tambara functors are morphisms of
presheaves of sets which commute with the norms and transfers. We write
$T^\naive(k)$ for the category of naive motivic Tambara functors.
\end{definition}

The product preservation condition just means that the canonical map
$M(X \coprod Y) \to M(X) \times M(Y)$ is an isomorphism, and that $M(\emptyset)
= *$; equivalently $M$ is a sheaf in the Zariski topology. By considering transfer
and norm along $X \coprod X \to X$, the set $M(X)$ acquires two binary operations
$+$ and $\times$. Considering $t: \emptyset \to X$, we have elements $0 =
tr_t(*)$ and $1 = N_t(*)$ which are units for the two binary operations. By
condition (3), $\times$ distributes over $+$. Consequently, $M$ is canonically a
sheaf of semirings. As before, the conditions imply that $f^*$ is a homomorphism
of semirings, $tr_f$ is a homomorphism of additive monoids, and $N_f$ is a
homomorphism of multiplicative monoids.

Let us note the following consequence of the axioms.

\begin{lemma}[projection formula] \label{lemm:projn-formula}
Let $M \in T^\naive(k)$ and $g: X \to Y \in \FEt_k$. Then for $a \in M(X), b \in
M(Y)$ we have $tr_g(a \cdot g^* b) = tr_g(a) \cdot b$.
\end{lemma}
\begin{proof}
This follows from the distributivity law applied to the exponential diagram
generated by $X \coprod Y \xrightarrow{g \coprod \id_Y} Y \coprod Y
\xrightarrow{\nabla} Y$, where $\nabla$ is the fold map, using the computation that
$R_f(X \coprod Y) \wequi X \times_Y Y \wequi X$ (recall that Weil restriction
along a fold map is just the product).
\end{proof}

\begin{definition}
We say that $M \in T^\naive(k)$ is \emph{group-complete} if the abelian semigroup
$(M(X), +)$ is an abelian group for all $X \in \FEt_k$. We write $T^\naive_{gc}(k)$
for the full subcategory of group-complete functors.
\end{definition}

\begin{theorem}[Tambara] \label{thm:tambara-gc}
Let $k$ be a connected scheme.

The inclusion $T^\naive_{gc}(k) \to T^\naive(k)$ has a left adjoint $M \mapsto M^+$, which
satisfies $M^+(X) = M(X)^+$, where $M(X)^+$ denotes the usual group-completion
of the abelian semigroup $(M(X), +)$.
\end{theorem}
\begin{proof}
A naive motivic Tambara functor is essentially the same as a semi-TNR functor
in the
sense of Tambara \cite[Section 2]{tambara1993multiplicative},
for the \emph{profinite} group $G=Gal(k)$. Tambara only treats finite groups,
but the extension to profinite groups is immediate. We spell out the details.

Let $L/k$ be a finite Galois extension with group $G$. Then the category $\Fin_G$ of finite
$G$-sets is a full subcategory of $\FEt_k$, by Grothendieck's Galois theory. The
restriction $M|_{\Fin_G}$ defines a semi-TNR functor. It follows from \cite[Theorem 6.1
and Proposition 6.2]{tambara1993multiplicative} that $M^+|_{\Fin_G}$ has a unique
structure of a TNR-functor such that the canonical map $(M \to M^+)|_{\Fin_G}$ is a
morphism of semi-TNR functors, and that $M^+|_{\Fin_G}$ is the universal map from
$M|_{\Fin_G}$ to a TNR-functor.

Now suppose that $L'/L/k$ is a bigger Galois extension, with group $G'$. Then
$(M|_{\Fin_{G'}})|_{\Fin_G} = M|_{\Fin_G}$ and consequently the norms on
$M^+|_{\Fin_{G'}}$
obtained by the above universal property, when further restricted to $\Fin_G$, coincide with
the norms obtained on $M^+|_{\Fin_G}$ directly. Now let $f: X \to Y \in \FEt_k$.
Then there exists a finite Galois extension $L/k$ with group $G$ such that
$f$ is in the image of $\Fin_G \to
\FEt_k$, and hence
we obtain a norm map $N_f$. By the above universal property, extending $L$ does
not change this norm, so in particular $N_f$ is well-defined independent of the
choice of $L$. This defines the structure of a naive motivic Tambara functor
on $M^+$, since all the required conditions can be checked after restriction to
$\Fin_G$ for varying $G$.

Let $A \in T^\naive_{gc}(k)$ and $F: M \to A$ be any morphism of Tambara functors.
Then there is a unique morphism of sheaves of additive abelian groups $F^+: M^+
\to A$. It remains to show that $F^+$ is a morphism of Tambara functors, i.e.
preserves norms. This can be checked on $F^+|_{\Fin_G}$ for varying $G$,
where it holds by Tambara's result. This concludes the proof.
\end{proof}

Let us also include for the convenience of the reader a proof of the following
well-known fact.

\begin{lemma} \label{lemm:gc-easy}
Let $A$ be an abelian group and $A_0 \subset A$ an abelian semigroup which
generates $A$ as an abelian group. Then the induced map $A_0^+ \to A$ is an
isomorphism.
\end{lemma}
\begin{proof}
It suffices to verify the universal property. Thus let $B$ be an abelian group.
If $f: A \to B$ is a homomorphism and $a \in A$, then there exist $a_1, a_2 \in
A_0$ with $a = a_1 - a_2$. Consequently $f(a) = f(a_1) - f(a_2)$ and $\Hom(A, B)
\to \Hom(A_0, B)$ is injective. To prove that $\Hom(A, B) \to \Hom(A_0, B)$ is
surjective, let $f_0 \in \Hom(A_0, B)$. Given $a \in A$, pick $a_1, a_2 \in A_0$
with $a = a_1 - a_2$, and put $f(a) = f_0(a_1) - f_0(a_2)$. I claim that this is
independent of the choices. Indeed if $a_1', a_2' \in A_0$ with $a_1' - a_2' =
a$, then $a_1 + a_2' = a_1' + a_2$, and hence $f_0(a_1) + f_0(a_2') = f_0(a_1 +
a_2') = f_0(a_1' + a_2) = f_0(a_1') + f_0(a_2)$, which implies the claim. From
this it easily follows that $f \in \Hom(A, B)$. This concludes the proof.
\end{proof}

We now investigate the localization of Tambara functors.

\begin{lemma} \label{lemm:norms-units}
Let $k$ be a connected scheme and $f: X \to Y \in \FEt_k$. Suppose $M \in
T^\naive_{gc}(k)$ and $n > 0$. Then $N_f(n \cdot 1_{M(X)}) \in (M(Y)[1/n])^\times$.
\end{lemma}
\begin{proof}
Via Grothendieck's Galois theory, we reduce to the analogous statement for
TNR-functors for some finite group $G$.
It is known that this category is symmetric monoidal, with initial object the Burnside
ring functor $A$. It is thus enough to prove this result for $A$, which is done
in \cite[Lemma 12.9]{bachmann-norms}.
\end{proof}

\begin{corollary} \label{corr:T3-loc}
Let $T^\naive_{gc}(k)[1/n]$ denote the full subcategory of $T^\naive_{gc}(k)$ on those
functors $M$ such that $n \in M(X)^\times$ for all $X \in \FEt_k$. Then the
inclusion $T^\naive_{gc}(k)[1/n] \to T^\naive_{gc}(k)$ has a left adjoint $M \mapsto
M[1/n]$, such that for $X \in \FEt_k$ we have $M[1/n](X) = M(X)[1/n]$.
\end{corollary}
\begin{proof}
We wish to make $M[1/n]$ into a Tambara functor by defining $N(x/n^k) =
N(x)/N(n)^k$, and $tr(x/n^k) = tr(x)/n^k$, and similarly for pullback.
Since transfer and pullback are additive, it is clear that they extend as
stated; the formula for the norm is well-defined by Lemma
\ref{lemm:norms-units}.
In order to check that this is a Tambara functor, the only difficulty is to
check that the distributivity law remains valid. Let $f: X \to Y \in \FEt_k$.
Note that that for $a \in M(X), b \in M(Y)$ we have $tr_f(a/n^k \cdot
f^*(b/n^l)) = tr_f(a \cdot f^*(b))/n^{k+l} = tr_f(a) b/n^{k+l} = tr_f(a/n^k)
b/n^l$ by definition and Lemma \ref{lemm:projn-formula}, i.e. the projection formula still
holds for $M[1/n]$.
Now let $A \xrightarrow{q} X
\xrightarrow{f} Y \in \FEt_k$ generate an exponential diagram, and $x \in M(A)$.
We have
$N_f(tr_q(x/n^k))=N_f(tr_q(x)/n^k) = N_f(tr_q(x))/N_f(n^k)$. Since $M$ satisfies
the distributivity law, this is the same as $(tr_{R_f(q)}N_pe^*(x))/N_f(n^k)$. 
Since $M[1/n]$ satisfies the projection formula (as noted above),
it remains to show that $N_p(n^k) = R_f(q)^*
N_f(n^k)$. This follows from the base change formula.

It is clear
that the canonical map $M \to M[1/n]$ is a morphism of Tambara functors, which
is the initial morphism to an object of $T^\naive_{gc}(k)[1/n]$. This concludes the
proof.
\end{proof}

Our main reason for studying naive motivic Tambara functors is that
they can be obtained by restriction from motivic Tambara functors of the first
kind. Indeed let $k$ be a field again and $M \in T^1_\scr{C}(k)$. Let $X \in
\scr C$. Then $\FEt_X \subset \scr C$ and by restriction we obtain $M|_{\FEt_X}
\in T^\naive(X)$. This observation allows us to reduce the following Corollary to
results about naive motivic Tambara functors.

\begin{corollary} \label{corr:T1-loc}
Let $\scr C \subset_\fet \Sm_k$ (where $k$ is again a perfect field), and assume that
$\scr C$ is closed under passing to summands (i.e. if $X \coprod Y \in \scr C$
then also $X \in \scr C$). For $n > 0$ denote by $T^1_\scr{C}(k)[1/n]$ the full
subcategory on those $M$ such that $U_1 M \in \HI_0(k)[1/n]$. Then the inclusion
$T^1_\scr{C}(k)[1/n] \to T^1_\scr{C}(k)$ has a left adjoint $M \mapsto M[1/n]$
which satisfies $U_1(M[1/n]) \wequi (U_1 M)[1/n]$.
\end{corollary}
\begin{proof}
As in the proof of Corollary \ref{corr:T3-loc}, the only difficulty is in
extending the norms to $M[1/n]$ and checking the distributivity law.
We need to prove that if $f: X \to Y \in \scr
C$ is finite étale, then $N_f(n \cdot 1_{M(X)}) \in (M(Y)[1/n])^\times$. Since
$M$ is a sheaf and $\scr C$ is closed under passing to summands, we may assume
that $Y$ is connected. In this case the result follows from
Lemma \ref{lemm:norms-units} applied to $M|_{\FEt_Y}$.
Now to verify the distributivity law, we may again restrict to $Y$ connected,
whence this follows from Corollary \ref{corr:T3-loc}, again applied to
$M|_{\FEt_Y}.$
\end{proof}

\section{Effective homotopy modules and sheaves with generalized transfers}
\label{sec:effective-homotopy-modules}

In this section we provide a more explicit description of the category
$\HI_0(k)$ of effective homotopy modules. A similar result (in the non-effective
case) was obtained by different means in \cite[Theorem
9.11]{ananyevskiy2017framed}.
We begin with some abstract preparation.

\begin{lemma} \label{lemm:t-str-mod}
Let $\scr C$ be a stable, presentably symmetric monoidal $\infty$-category.
Suppose given an accessible $t$-structure on $\scr C$ such that
that $\scr C_{\ge 0} \otimes \scr C_{\ge 0} \subset \scr C_{\ge 0}$. Let
$A \in \CAlg(\scr C_{\ge 0})$. Then:
\begin{enumerate}
\item The $\infty$-category $A\Mod$ has a unique $t$-structure such that $U:
  A\Mod \to \scr C$ is $t$-exact. This $t$-structure is accessible.
\item If $\pi_0(\1) \to \pi_0(A)$ is an isomorphism (in $\scr C^\heart$), then
  $U^\heart: A\Mod^\heart \to \scr C^\heart$ is an equivalence.
\end{enumerate}
\end{lemma}
\begin{proof}
The $t$-structure is unique if it exists, since $U$ is conservative. We show
existence.
Since the $t$-structure on $\scr C$ is accessible, $\scr C_{\ge 0}$ is
presentable, and hence there exists a set of objects $P \subset \scr C_{\ge 0}$
generating $\scr C_{\ge 0}$ under colimits. Denote by $F: \scr C \to A\Mod$ the
left adjoint of $U$ and write $A\Mod_{\ge 0}$ for the full subcategory of
$A\Mod$ generated under colimits and extensions by $FP$. Then $A\Mod_{\ge 0}$
is the non-negative part of an accessible $t$-structure on $A\Mod$
\cite[Proposition 1.4.4.11]{lurie-ha}. It remains to show that $U$ is $t$-exact.
By construction $F$ is right $t$-exact, so $U$ is left $t$-exact. We thus need
to show that $U(A\Mod_{\ge 0}) \subset \scr C_{\ge 0}$. Since $U$ preserves
colimits \cite[Corollary 4.2.3.7]{lurie-ha} and extensions,
for this it is enough to show that $UFP
\subset \scr C_{\ge 0}$. But for $X \in P$ we have $UFX = X \otimes A \in \scr
C_{\ge 0}$ by assumption. This proves (1).

To prove (2), consider the induced adjunction $F^\heart: \scr C^\heart
\leftrightarrows A\Mod^\heart: U^\heart$. Since $U$ is $t$-exact and
conservative, we need only show that for $X \in \scr C^\heart$ the canonical map
$X \to UF^\heart X = (UFX)_{\le 0}$ is an equivalence. We have the triangle
$A_{> 0} \to A \to A_{\le 0}$ giving us $X \otimes A_{> 0} \to X \otimes A = UFX
\to X \otimes A_{\le 0}$. By assumption, $X \otimes A_{> 0} \in \scr C_{> 0}$
and hence $(UFX)_{\le 0} \wequi (X \otimes A_{\le 0})_{\le 0}$. But by
assumption $A_{\le 0} \wequi \1_{\le 0}$, and so reversing the steps with $\1$
in place of $A$ we similarly find that $(X \otimes \1_{\le 0})_{\le 0} \wequi (X \otimes
\1)_{\le 0} \wequi X$. This concludes the proof.
\end{proof}

We recall also the following well-known result.

\begin{lemma} \label{lemm:module-detect}
Let $F: \scr C \to \scr D$ be a symmetric monoidal functor of stable, compact-rigidly
generated, presentably symmetric monoidal $\infty$-categories. Assume that $F$
preserves colimits and has dense image.
Then $F$ has a lax symmetric monoidal right adjoint $U$, so
$U(\1_\scr{D}) \in \CAlg(\scr C)$, and $U$ induces an equivalence $\scr D \wequi
U(\1_\scr{D})\Mod$.
\end{lemma}
We note that if $F: \scr C \to \scr D$ is a symmetric monoidal functor between
stable, presentably symmetric monoidal $\infty$-categories which preserves colimits and has dense
image, and $\scr C$ is compact-rigidly generated, then $\scr D$ is
compact-rigidly generated as soon as $\1_\scr{D}$ is compact.
\begin{proof}
The existence of $U$ follows from the adjoint functor theorem,
and the factorization $\scr D \to U(\1_\scr{D})\Mod$ is also obtained by abstract
nonsense \cite[Construction 5.23]{mathew2017nilpotence}. Note that
$U(\1_\scr{D})\Mod$ satisfies the same assumptions as $\scr C$. In other words
we may assume that $U\1 \wequi \1$. Now apply \cite[Lemma 21]{bachmann-hurewicz}. %XXX This reference will need updating.
\end{proof}

Recall the category of presheaves with generalized transfers \cite[Section 4]{calmes2014finite}. We
write $\widetilde{\HI}(k)$ for the category of homotopy invariant
Nisnevich sheaves with generalized transfers. Recall also the canonical
equivalence $\widetilde{\HI}(k) \wequi \widetilde{\DM}^\eff(k)^\heart$ from
\cite[Corollary 3.2.11]{Deglise17}. Now we come to our identification result
for the category of effective homotopy modules.

\begin{theorem} \label{thm:HI0-gen-transfer}
Let $k$ be an infinite perfect field, $char(k) \ne 2$. The adjunction
\[ \widetilde M: \SH(k) \leftrightarrows \widetilde{\DM}(k): U \]
induces an equivalence $\SH(k)^\heart \wequi \widetilde{\DM}(k)^\heart$ identifying
the full subcategories $\SH(k)^{\eff\heart}$ and
$\widetilde{\DM}^\eff(k)^\heart$. In particular
\[ \widetilde M^{\eff\heart}: \HI_0(k) \wequi \SH(k)^{\eff\heart}
          \to \widetilde{\DM}^\eff(k)^\heart \wequi \widetilde{\HI}(k) \]
is an equivalence.
\end{theorem}
\begin{proof}
We first prove that $\widetilde M^\heart: \SH(k)^\heart \leftrightarrows
\widetilde{\DM}(k)^\heart: U^\heart$ is an adjoint equivalence. Since $U^\heart$
is conservative, it suffices to show that for $E \in \SH(k)^\heart$ the
canonical map $\alpha: E \to U^\heart M^\heart E$ is an equivalence. Since $2
\ne e$, it suffices
to prove that $\alpha[1/2]$ and $\alpha[1/e]$ are equivalences.

Recall that $U(\1) \in \SH(k)_{\ge 0}$ and $\1 \to U(\1)$ induces $\1_{\le 0}
\wequi U(\1)_{\le 0}$. This follows from the main result in
\cite{calmes2017comparison}, and the cancellation theorem for $\widetilde{\DM}$
\cite{FaselCancellation}. Under our assumptions, $\SH(S)[1/e]$ is
compact-rigidly generated \cite[Corollary B.2]{levine2013algebraic}, and hence
so is $\widetilde{\DM}(k, \Z[1/e])$ (being compactly generated by \cite[Proposition
3.2.21]{Deglise17} and the cancellation theorem). It hence
follows from Lemma \ref{lemm:module-detect} that
$\widetilde{\DM}(k, \Z[1/e])$ is equivalent to the category of modules over
$U(\1)[1/e]$, and hence $\widetilde{\DM}(k)^\heart[1/e] \wequi \SH(k)^\heart[1/e]$
by Lemma \ref{lemm:t-str-mod}. It follows that $\alpha[1/e]$ is an equivalence.

Now consider $\alpha[1/2]$. If $e = 1$ then $\alpha[1/2]$ is an equivalence,
since $\alpha = \alpha[1/e]$ is. Thus we may assume that $e > 1$. In this case
$W(k)[1/2] = 0$, so $\SH(k)[1/2] = \SH(k)^+$ and similarly
$\widetilde{\DM}(k)[1/2] = \widetilde{\DM}(k)^+ \wequi \DM(k)$
\cite[Theorem 5.0.2]{Deglise17}.
The functor $\DM(k)^\heart \to \SH(k)^\heart$ is fully faithful with essential
image the subcategory $\SH(k)^{\heart,\eta=0}$ of those objects on which the
motivic Hopf element $\eta$ acts as zero \cite{deglise-htpy-modules}. Consequently
$M^\heart U^\heart E[1/2]^+ = E[1/2]^+/\eta = E[1/2]^+$, and so $\alpha[1/2] = \alpha[1/2]^+$
is an equivalence. We have thus shown that $\alpha$ is an equivalence.

The adjunction $i: \SH(k)^\eff \leftrightarrows \SH(k): f_0$ induces $i^\heart:
\SH(k)^{\eff\heart} \leftrightarrows \SH(k)^\heart: f_0^\heart$, and $i^\heart$
is fully faithful \cite[Proposition 5(2)]{bachmann-very-effective}. The same
argument applies to $\widetilde{DM}$. The diagram
\begin{equation*}
\begin{CD}
\SH(k)^\eff @>\widetilde M^{\eff}>> \widetilde{\DM}^\eff(k) \\
@ViVV                     @ViVV                  \\
\SH(k)        @>\widetilde M>>      \widetilde{\DM}(k) \\
\end{CD}
\end{equation*}
commutes (by definition), and hence so does the induced diagram
\begin{equation*}
\begin{CD}
\SH(k)^{\eff\heart} @>\widetilde M^{\eff\heart}>> \widetilde{\DM}^\eff(k)^\heart \\
@V{i^\heart}VV                               @V{i^\heart}VV             \\
\SH(k)^\heart        @>\widetilde M^\heart>>      \widetilde{\DM}(k)^\heart. \\
\end{CD}
\end{equation*}
It follows that $\widetilde M^\heart$ maps the full subcategory $\SH(k)^{\eff\heart}$ of
$\SH(k)^\heart$ into the full subcategory $\widetilde{\DM}^\eff(k)^\heart$ of
$\widetilde{\DM}(k)^\heart$. Since $\widetilde{\DM}^\eff(k)^\heart$ is generated
under colimits by the (truncated) motives of varieties and $M^\heart$ is an
equivalence, so preserves subcategories closed under colimits, we conclude that
$\widetilde M^\heart(\SH(k)^{\eff\heart}) = \widetilde{\DM}^\eff(k)^\heart$. In other
words, $M^{\heart,\eff}$ is essentially surjective.
This concludes the proof.
\end{proof}

\begin{remark} \label{rmk:infinite}
Note that we do not claim that the inverse of $\widetilde M^{\eff\heart}:
\SH(k)^{\eff\heart} \to \widetilde{\DM}^\eff(k)^\heart$ is given by
$U^{\eff\heart}$. Indeed I do not not know if
$U(\widetilde{\DM}^\eff(k)) \subset \SH(k)^\eff$. This is true after inverting
the exponential characteristic, by \cite[Corollary 5.1 and Lemma
5.3]{bachmann-criterion}. \thought{Might it be possible to prove directly that
$\Z(X_+ \wedge \Gm)$ is rationally contractible?}
\end{remark}

\begin{remark}
I am confident that this result is true even if $k$ is finite. The most
natural way to prove this would be to extend the results about
$\widetilde{\DM}(k)$ to finite fields. I am confident that this can be done
using the methods of \cite[Appendix B]{EHKSY}, but this would take us too far
afield. In the sequel we will treat finite fields by using an alternative
description of $\HI_0(k)$ in terms of framed transfers \cite[Theorem
5.14]{bachmann-moving}.
\end{remark}

The following corollary is the main reason we need the above result. It allows
us to write down generators for the abelian group $(\E X)(K)$.

\begin{corollary} \label{corr:generation}
Let $k$ be a perfect field of exponential characteristic $e \ne 2$.

Let $K/k$ be a field extension and $X \in \Sm_k$ and assume that $K$ is perfect.
Then $(\E
X)(K)$ is generated as an abelian group by expressions of the form $tr_f g^*
\id_X$, where $\id_X \in (\E X)(X)$ corresponds to the identity morphism, $f:
Spec(L) \to Spec(K)$ is a finite (hence separable) field extension, and $g: Spec(L) \to
X$ is any morphism.
\end{corollary}
We shall give two proofs: one assuming that $k$ is infinite and relying on
Theorem \ref{thm:HI0-gen-transfer}, and another one that works in general.
\begin{proof}[Proof assuming $k$ infinite.]
Since our categories are additive, we may assume that $X$ is connected.

By Theorem \ref{thm:HI0-gen-transfer},
it is enough to prove the claim for $\ul{h}_0(\widetilde MX)$, where
$\widetilde M: \Sm_k \to \widetilde{\DM}^\eff(k)$ is the canonical functor. By
\cite[Corollary 3.2.14]{Deglise17} we have $\widetilde MX = L_{Nis} Sing_*
a_{Nis} \tilde{c}(X)$. Consequently $\ul{h}_0(\widetilde MX)(K)$ is a
quotient of $\tilde{c}(X)(K) = \Hom_{\widetilde{Cor}_k}(K, X)$. By
\cite[Example 4.5]{calmes2014finite}, up to non-canonical isomorphism we have
\[ \Hom_{\widetilde{Cor}_k}(K, X) \wequi \bigoplus_{x \in (X_K)_{(0)}} GW(K(x)), \]
where $(X_K)_{(0)}$ is the set of closed points.
The class $\alpha \in \Hom{\widetilde{Cor}_k}(K, X)$
corresponding to an element $a \in GW(K(x))$ and $f: Spec(K(x)) \to X$ is
$\alpha = tr_{K(x)/K}(u_x \cdot a \cdot f^*\id_X)$, where $u_x \in GW(K(x))$ is some unit
reflecting the non-canonicity of the above isomorphism.

Since $char(k) \ne 2$,
$GW(K(x))$ is generated as an abelian group
by elements of the form $tr_{L/K(x)}(1)$ with $L/K(x)$ finite separable
\cite[paragraph before Proposition 22]{bachmann-gwtimes}.
It follows from this and the base change formula \cite[Proposition 10]{bachmann-real-etale}
that $\ul{h}_0(\widetilde MX)$ is generated by elements of the form
\[ tr_{K(x)/K}(tr_{L/K(x)}(1) f^* \id_X)
     = tr_{K(x)/K}tr_{L/K(x)} ((Spec(L) \to Spec(K(x)))^* f^* \id_X), \]
as needed. This concludes the proof.
\end{proof}

\begin{proof}[Proof for general $k$.]
We use \cite[Theorem 5.14]{bachmann-moving}, which tells us that $\HI_0(k)$ is
equivalent to the category of homotopy invariant, ``stable'' sheaves with
``equationally framed transfers''. The main upshot for us is that $(\E X)(K)$ is
generated by elements of the form $\alpha^*(\id_X)$, where $\alpha: Spec(K)
\rightsquigarrow X$ is a ``framed correspondence''.
This consists, among other things, of a scheme $Z$ finite over $K$ and a map
$\alpha': Z \to X$. Let $Spec(L) = Z_\red$, and write $g$ for the composite
$Z_\red \hookrightarrow Z \xrightarrow{\alpha'} X$. Then by \cite[Lemma
5.16]{bachmann-moving}, we have $\alpha^*(\id_X) = tr_{L/K}(c_\alpha g^* \id_X)$, where
$c_\alpha \in GW(L)$ is a certain class determined by $\alpha$.

The rest of the proof proceeds as before.
\end{proof}

\begin{remark} \label{rmk:non-perfect-gens}
The only reason above to assume that $K$ is perfect is that then the finite
extension $f: Spec(L) \to Spec(K)$ is automatically étale, and hence we have a
transfer morphism as discussed previously. In fact, as long as $char(k) \ne 2$,
for any finite (but not necessarily separable) field extension, and any homotopy
module $M$, there exist the \emph{cohomological transfer morphism} $tr_f: M(L)
\to M(K)$ \cite[Section 4.3]{A1-alg-top}; it coincides with the previous
transfer if $f$ is étale \cite[Lemma 2.3]{calmes2014finite}.
The above corollary remains true as stated for imperfect $L$, provided that all
finite (not necessarily separable) extensions are considered, and $tr_f$ denotes
the cohomological transfer.
\end{remark}

\section{Normed effective homotopy modules I: construction and basic properties}
\label{sec:normed-homotopy-modules-I}

In this section we construct a final category of motivic Tambara functors, this
time as a category of normed spectra.

The functor \[ \SH^\otimes: \Span(\Sch, \all, \fet) \to \widehat{\Cat}_\infty, X
\mapsto \SH(S) \] has a full subfunctor \[ \SH^{\veff\otimes}: \Span(\Sch, \all,
\fet) \to \widehat{\Cat}_\infty, X \mapsto \SH(X)^\veff. \] Moreover there is a
natural transformation
\[ \tau_{\le 0}^\eff: \SH^{\veff\otimes} \Rightarrow \SH^{\eff\heart},
   \SH(X)^\veff \ni E \mapsto \tau_{\le 0}^\eff(E) \in\SH(X)^{\eff\heart} \]
of functors on $\Span(\Sch, \all, \fet)$.
This is constructed in \cite[before Proposition 13.3]{bachmann-norms}.

\begin{proposition} \label{prop:SH-veff-norms}
Let $f: Y \to X \in \Sch$. The functors $f^*$, $f_\#$ for $f$ smooth and
$f_\otimes$ for $f$ finite étale preserve $\SH(\bullet)^\veff$.
\end{proposition}
\begin{proof}
The claims about $f^*$ and $f_\otimes$ are already implicit in the existence of
$\SH^{\veff\otimes}$. Since $f_\#$ preserves colimits (and hence extensions), it
is enough to show that $f_\# \Sigma^\infty_+ U \in \SH(X)^\veff$ for $U \in
\Sm_Y$, which is clear.
\end{proof}

\begin{proposition} \label{prop:nalg-nullification}
Let $f: Y \to X \in \Sch$. If $f$ is smooth, the functor $f^*:
\SH(Y)^{\eff\heart} \to \SH(X)^{\eff\heart}$ has a left adjoint still denoted
$f_\#$.

The functors $\tau_{\le 0}^\eff: \SH(\bullet)^\veff \to
\SH(\bullet)^{\eff\heart}$ commute with $f^*$, with $f_\#$ for $f$ smooth, and
with $f_\otimes$ for $f$ finite étale.
\end{proposition}
\begin{proof}
The statements about $f^*$ and $f_\otimes$ are already implicit in the existence
of the natural transformation $\tau_{\le 0}^\eff$ of functors on $\Span(\Sch,
\all, \fet)$.

The functor $f_\#: \SH(X)^\veff \to \SH(Y)^\veff$ preserves the subcategory
$\SH(\bullet)^\eff_{\ge 1}$. By adjunction it follows that $f^*: \SH(Y)^\veff
\to \SH(X)^\veff$ preserves $\SH(\bullet)^{\eff\heart}$. From this it is easy to
check directly that the composite \[ \SH(Y)^{\eff\heart} \to \SH(Y)^\veff
\xrightarrow{f_\#} \SH(X)^\veff \xrightarrow{\tau_{\le 0}^\eff}
\SH(X)^{\eff\heart} \] is left adjoint to $f^*: \SH(Y)^{\eff\heart} \to
\SH(X)^{\eff\heart}$.
\end{proof}

\begin{definition}
Let $\scr C \subset_\fet \Sch_S$. We denote the full subcategory of
$\NAlg_\scr{C}(\SH)$ consisting of those normed spectra with underlying spectrum
in $\SH(S)^{\eff\heart}$ by $\NAlg_\scr{C}(\SH(S)^{\eff\heart})$ and call
it the category of $\scr C$-normed effective homotopy modules. If $S = Spec(k)$
is the spectrum of a field, so $\SH(S)^{\eff\heart} \wequi \HI_0(k)$, then we
also denote $\NAlg_\scr{C}(\SH(S)^{\eff\heart})$ by $\NAlg_\scr{C}(\HI_0(k))$.
\end{definition}

\begin{lemma} \label{lemm:NAlg-HI-adjoint}
The functor $U: \NAlg_\scr{C}(\HI_0(S)) \to \HI_0(S)$
preserves limits and sifted colimits.

If $\scr C = \SmQP_S$, $U$ has a left adjoint $F$ which satisfies
\[ UF E \wequi \colim_{\substack{f\colon X\to S\\p\colon Y\to X}} \tau_{\le
0}^\eff f_\# p_\otimes(E_Y). \]
Here the colimit is over the source of the cartesian fibration classified by
$\SmQP_S^\op \to \scr S$, $X\mapsto\FEt_X^\simeq$.

Moreover for $(f: X \to S, p:
Y \to X)$ in the indexing category, the canonical map $\tau_{\le 0}^\eff f_\# p_\otimes(E_Y) \to UF
E$ is induced by the composite
\[ f_\# p_\otimes E_Y \to f_\# p_\otimes (UFE)_Y \to f_\# (UFE)_X \to UF E, \]
where the first map is induced by the unit map $E \to UF E $, the second map is
induced by the multiplication $p_\otimes (FE)_Y \to (FE)_X$, and the third map
is a co-unit map.
\end{lemma}
\begin{proof}
The claim about limits and colimits follows from \cite[Remark
7.7]{bachmann-norms}, using Proposition \ref{prop:nalg-nullification}.

The functor $\NAlg_{\SmQP_S}(\SH) \to \SH(S)$ has a left adjoint $\bar{F}$ given by the same
colimit as in the claim, but without the $\tau_{\le 0}^\eff$ \cite[Remark
16.25]{bachmann-norms}. If $E \in
\SH(S)^\veff$ then $U\bar{F} E \in \SH(S)^\veff$, since the latter category is
closed under colimits and $f_\# p_\otimes g^*$ by Proposition
\ref{prop:SH-veff-norms}. Hence $\bar{F}$ restricts to a functor $\SH(S)^\veff \to
\NAlg(\SH(S)^\veff)$ left adjoint to $\NAlg(\SH(S)^\veff) \to
\SH(S)^\veff$. By \cite[Proposition 13.3]{bachmann-norms}, the inclusion
$\NAlg(\SH(S)^{\eff\heart}) \to \NAlg(\SH^\veff(S))$ has a left adjoint pointwise given by
$\tau_{\le 0}^\eff$. Since $\tau_{\le 0}^\eff$ preserves colimits and commutes
with $f_\#p_\otimes g^*$ by Proposition \ref{prop:nalg-nullification}, the
result follows.
\end{proof}

\begin{remark} \label{rmk:1-categorical}
Under the conditions of Lemma \ref{lemm:NAlg-HI-adjoint}, the category
$\NAlg_\scr{C}(\SH(S)^{\eff\heart})$ is monadic over the 1-category
$\SH(S)^{\eff\heart}$, and hence a 1-category.
\end{remark}

\section{Normed effective homotopy modules II: main theorem}
\label{sec:normed-homotopy-modules-II}

In this section we put everything together: we show that
$\NAlg_\scr{C}(\HI_0(k))$ is
equivalent $T^2_\scr{C}(k)$, for an appropriate $\scr C$.
In fact from now on, we set $\scr C = \SmQP_k$,
$\scr V = \SmQP$ and suppress both from the
notation. In particular we have $T^2(k) \wequi T^1(k)$, by Proposition \ref{prop:comparison}.

\begin{construction}
If $E \in \NAlg(\HI_0(k))$, then $E$ naturally defines an object of $T^2(k)$
\cite[Proposition 7.19(1), Lemma 7.20]{bachmann-norms}.
We denote the resulting functor by $\rho: \NAlg(\HI_0(k)) \to
T^2(k)$.
\end{construction}

Note that the following diagram commutes
\begin{equation*}
\begin{CD}
\NAlg(\HI_0(k)) @>{\rho}>> T^2(k) \\
@VUVV                       @VU_2VV \\
\HI_0(k) @= \HI_0(k).
\end{CD}
\end{equation*}
From now on, let $e$ be the exponential characteristic
of $k$. Write $\NAlg(\HI_0(k))[1/e]$ for the full subcategory on those $M \in \NAlg(\HI_0(k))$
such
that $UM \in \HI_0(k)[1/e]$. Define $T^2(k)[1/e]$ similarly. Then the inclusion
$\NAlg(\HI_0(k))[1/e] \to \NAlg(\HI_0(k))$ has a left adjoint $M \mapsto M[1/e]$
such that $U(M[1/e]) = U(M)[1/e]$ \cite[Proposition 12.6]{bachmann-norms}.
Similarly for the inclusion $T^2(k)[1/e] \to T^2(k)$, by Proposition
\ref{prop:comparison} and Corollary \ref{corr:T1-loc}.
\begin{theorem} \label{thm:main}
Let $k$ be a perfect field of exponential characteristic $e \ne 2$.
Let $X \in \SmQP_k$. The canonical map $\E X[1/e] \to UF \E X[1/e] \wequi U_2
\rho F \E X[1/e]$
exhibits $\rho F \E X[1/e]$ as the free $e$-local (i.e. in $T^2(k)[1/e]$)
motivic Tambara functor of the second kind on $\E
X[1/e]$. In other words, for $A \in T^2(k)[1/e]$, the canonical map
\[ u: [\rho F \E X[1/e], A]_{T^2(k)} \to [U_2 \rho F \E X[1/e][1/e], U_2 A]_{\HI_0(k)} \to [\E X[1/e], U_2 A]_{\HI_0(k)} \]
is an isomorphism.
\end{theorem}
\begin{proof}

\textbf{Preliminary remarks.}
By Lemma \ref{lemm:NAlg-HI-adjoint} we have
\[ UF \E X = \colim_{\substack{f\colon S\to Spec(k)\\p\colon Y\to S}}
                          \tau_{\le 0}^\eff f_\# p_\otimes((\E X)_Y). \]
Note that $\E X = \tau_{\le 0}^\eff \Sigma^\infty_+ X$. By Lemma
\ref{prop:nalg-nullification}, $f_\#, p_\otimes, g^*$ commute with the
localization functor $\tau_{\le 0}^\eff$. Consequently we find that
\[ UF \E X = \colim_{\substack{f\colon S\to Spec(k)\\p\colon Y\to S}}
                          \E R_p (X \times_k Y) \]
and hence
\[ UF \E X[1/e] = \colim_{\substack{f\colon S\to Spec(k)\\p\colon Y\to S}}
                          \E R_p (X \times_k Y) [1/e]. \]
From now on, we will suppress $[1/e]$ from the notation; it should be understood
that all homotopy modules in sight are in $\HI_0(k)[1/e]$, and if not should be
replaced by $(?)[1/e]$.
We will also write $X_Y$ for $X \times_k Y$, particularly when viewed as a
$Y$-scheme, and similarly for other pairs of schemes.
It follows that a morphism of effective homotopy modules $\alpha \in [UF \E X,
A]$ consists of the following data:
\begin{itemize}
\item For each $S \in \SmQP_k$ and each finite étale morphism $p: Y \to S$, a
  class $\alpha_p \in A(R_p (X_Y))$,
\end{itemize}
subject to the following compatibility condition:
\begin{itemize}
\item For every cartesian square
\begin{equation} \label{eq:cart-square-condn}
\begin{CD}
Y_1 @>k>> Y_2 \\
@Vp_2VV   @Vp_1VV \\
S_1 @>h>> S_2
\end{CD}
\end{equation}
  in $\SmQP_k$ with $p_1$ (and hence $p_2$) finite étale, denote by $q:
  R_{p_1}(X_{Y_1}) \to R_{p_2}(X_{Y_2})$ the canonical map induced by $k$.
  Then $q^*(\alpha_{p_2}) = \alpha_{p_1} \in A(R_{p_1}(X_{Y_1}))$.
\end{itemize}

We have in particular the class $\alpha_1 := \alpha_{\id_k} \in A(X)$.

\textbf{Injectivity of $u$.} Let us
first show that if $\alpha: UF\E X \to A$ is indeed a morphism in $T^2(k)$, then
the classes $\alpha_p$ are all determined by $\alpha_1$. In other words, we will
show that $u$ is injective. To do this, let for $S \in \SmQP_k$ and $p: Y \to S$
finite étale, $a_p \in (UF\E X)(R_p(X_Y))$ denote the class corresponding
to the canonical map $\E R_p(X_Y) \to UF\E X$ coming from the colimit formula.
Then $\alpha_p = \alpha(a_p) \in
A(R_p(X_Y))$. Let $r_Y: X_Y \to X$ denote the canonical projection. It
follows from the ``moreover'' part of Lemma \ref{lemm:NAlg-HI-adjoint}
that
\begin{equation} N_{p,X_Y}(r_Y^*(a_1)) = a_p. \label{eq:ap-a1} \end{equation}
Thus $\alpha_p =
\alpha(a_p) = \alpha(N_{p,X_Y}(r_Y^*(a_1))) = N_{p,X_Y} r_Y^* \alpha(a_1) =
N_{p,X_Y} r_Y^* \alpha_1$, since $\alpha$ was assumed to be a morphism of
Tambara functors. This proves that $u$ is injective.

\textbf{Surjectivity of $u$.}
Now let $\alpha_1 \in A(X)$ and put $\alpha_p = N_{p, X_Y} r_Y^*(\alpha_1)$.
Consider a cartesian square as in condition \eqref{eq:cart-square-condn}. Then
\[ q^*(\alpha_{p_2}) = q^* N_{p_2,X_{Y_2}} r_{Y_2}^* (\alpha_1)
      = N_{p_1 X_{Y_1}} r_{Y_1}^* (\alpha_1) = \alpha_{p_1}, \]
by condition (2) of Definition \ref{defn:tambara-2}. In other words, the
compatibility condition is satisfied and we obtain a morphism of homotopy
modules $\alpha: UF\E X \to A$ corresponding to $\alpha_1$.

What remains to be
done is to show that this is a morphism of Tambara functors; then $u$ will be
surjective. By Lemma \ref{lemm:T1-gen-full} (and Proposition
\ref{prop:comparison}), it is enough to prove that if $K/k$ is
the perfect closure of a finitely
generated field extension, and $p: Spec(L) \to Spec(K)$ is finite étale, then
the following square commutes
\begin{equation} \label{eq:target-square}
\begin{CD}
(F\E X)(L) @>\alpha>> A(L) \\
@VN_pVV               @VN_pVV \\
(F\E X)(K) @>\alpha>> A(K).
\end{CD}
\end{equation}

We shall do this as follows. Let, for each $L \in \FEt_K$, $(F \E X)_0(L) \subset (F \E X)(L)$ denote the
subset of those elements obtained by iterated application of norm, restriction
and transfer (all along finite étale morphisms) from elements $b \in (F \E
X)(L')$, corresponding to maps of schemes $b': Spec(L') \to X$; in other words $b
= b'^*(a_1)$. We shall prove the following:
\begin{enumerate}[(a)]
\item If $s_1 \in (F \E X)_0(L_1)$ and $s_2 \in (F \E X)_0(L_2)$, then $(s_1,
s_2) \in (F
  \E X)_0(L_1 \coprod L_2)$ (for any $L_1, L_2 \in \FEt_K$).
\item The $\Z[1/e]$-module $(F \E X)(L)$ is generated by $(F \E X)_0(L)$ (for
  any $L \in \FEt_K$).
\item For any $p: L \to K \in \FEt_K$ and $a \in (F \E X)_0(L)$, we have $N_p
  (\alpha(a)) = \alpha(N_p(a))$.
\end{enumerate}

Let $F \E X|_{\FEt_K} \in T^\naive_{gc}(k)[1/e]$ denote the induced naive
motivic Tambara functor.
By (a), the subfunctor $(F \E X)_0 \subset F \E X|_{\FEt_K}$ preserves finite
products, and hence $(F \E X)_0 \in T^\naive(k)$. By (b), the canonical map $(F
\E X)_0^+[1/e] \to F \E X|_{\FEt_K}$ is an isomorphism (e.g. use Lemma
\ref{lemm:gc-easy}).
By (c), the composite $(F\E X)_0 \to F\E
X|_{\FEt_K} \xrightarrow{\alpha} A|_{\FEt_K}$ is a morphism of naive motivic
Tambara functors. It follows
that the unique induced map $(F\E X)_0^+[1/e] \wequi F\E
X|_{\FEt_K} \to A|_{\FEt_K}$ compatible with the $\Z[1/e]$-module structures, is a
morphism of naive motivic Tambara functors. Since $\alpha|\FEt_K$ is
compatible with the $\Z[1/e]$-module structures ($\alpha$ being a morphism of
homotopy modules), it must be this unique map, and hence a morphism of naive
motivic Tambara functors. This proves that square \eqref{eq:target-square}
commutes. This concludes the proof, modulo establishing (a)-(c).

\textbf{Proof of (a).}
We may assume that $s_1$ is obtained by a sequence of
operations $O^{(e_1)}_{f_1} \dots O^{(e_n)}_{f_n}(x_1)$, where $f_i$ are finite étale
maps, $e_i \in \{1, 2, 3\}$, $O^{(1)}$ means pullback, $O^{(2)}$ means norm,
$O^{(3)}$ means transfer, and $x_1$ corresponds to a map of schemes $x_1: L_1'
\to X$. Similarly $s_2 =
O^{(e_1')}_{f_1'} \dots O^{(e_m')}_{f_m'}(x_2)$. Then
\[ (s_1, s_2) = O^{(e_1)}_{f_1 \coprod \id} \dots O^{(e_n)}_{f_n \coprod \id}
O^{(e_1')}_{\id \coprod f_1'} \dots O^{(e_m')}_{\id \coprod f_m'}(x_1, x_2), \]
so it is enough to show that $(x_1, x_2) \in (F \E X)_0(L_1' \coprod L_2')$.
This is clear.

\textbf{Proof of (b).}
The $\Z[1/e]$-module $\E R_p X_Y (L)$ is generated by transfers of pullbacks of
$a_p$, by Corollary \ref{corr:generation}. Consequently $F \E X (L)$ is
generated as a $\Z[1/e]$-module by transfers and norms of pullbacks of $a_1$, by
\eqref{eq:ap-a1}. This was to be shown.

\textbf{Proof of (c).}
Let us call a section $s \in (F\E X)(L)$ \emph{good} if for any span $Spec(K')
\xleftarrow{p'} Spec(L') \xrightarrow{f} Spec(L)$ with $p'$ finite étale, we have
$\alpha(N_{p'}f^*(s)) = N_{p'}f^*\alpha(s)$. We need to show that all sections
of $(F \E X)_0$ are good. We shall prove that good sections are closed under
norms, transfer and pullback (steps (i)-(iii) below), and that sections of the
form $b^* a_1$ are good (step (iv)). This implies the desired result.

Step (i). Suppose $s$ is good and $f: Spec(L') \to Spec(L)$ is arbitrary. Then
$f^* s$ is good. This follows from transitivity of pullback.

Step (ii). Suppose $b: Spec(L') \to Spec(L)$ is finite étale and $s \in (F\E
X)(L')$ is good. Then $N_b(s)$ is good. Indeed by the base change formula,
we may assume that $f = \id$ and $p'=p$. Then $N_p(N_b(s)) = N_{p \circ b}(s)$
and so the relevant equality holds by assumption.

Step (iii). Suppose $b: Spec(L') \to Spec(L)$ is finite étale and $s \in (F\E
X)(L')$ is good. Then $tr_b(s)$ is good. By the base change
formula, we may assume that $f = \id$ and $p'=p$. Now $N_p tr_b (s) = tr_? N_?
e^*(s)$, by the distributivity law.
Since $N_? e^*(s)$ is good by steps (i) and (ii), and any morphism of
homotopy modules commutes with transfers, $tr_b(s)$ is indeed good.

Step (iv). Suppose that $s$ corresponds to a map of schemes $b: Spec(L) \to X$.
Then $s$ is good. Note that
$f^*(s)$ corresponds to $Spec(L') \to Spec(L) \to X$, so we may assume that $f =
\id$ and $p' = p$. Note that $s = b^* a_1$.
Let $\tilde b: Spec(L) \to X_L$ be induced by $b$, and write $b': Spec(K) \to
R_p(X_L)$ for the Weil restriction of $\tilde b$ along $p$. We shall use the following
observation, proved below: If $B \in T^2(k)$ is arbitrary and $t \in B(X)$, then
\begin{equation}\label{eq:technical-N-trick}
   N_p b^* t = b'^*N_{p,X_L}r_L^*(t).
\end{equation}
Thus
\[ \alpha(N_p b^* a_1) = \alpha(b'^* N_{p,X_L} r_L^* a_1) = b'^* \alpha(a_p) =
     b'^* \alpha_p,\]
using \eqref{eq:technical-N-trick} with $t = a_1$ and \eqref{eq:ap-a1}.
On the other hand
\[ N_p(\alpha(b^* a_1)) = N_p b^* \alpha_1 = b'^* N_{p,X_L} r_L^* \alpha_1 =
b'^* \alpha_p \] as well, using  \eqref{eq:technical-N-trick} with $t =
\alpha_1$, and \eqref{eq:ap-a1} again. Hence $s$ is good.

\textbf{Proof of (6).}
Apply condition (2) of Definition \ref{defn:tambara-2} to the diagram
\begin{equation*}
\begin{CD}
Spec(L) @>\tilde{b}>> X_L \\
@|    @VVV \\
Spec(L) @=   Spec(L)   \\
@VpVV  @VpVV \\
Spec(K) @= Spec(K).
\end{CD}
\end{equation*}
The result follows since $R_p(\tilde b) = b'$ and the composite $L
\xrightarrow{\tilde{b}} X_L \xrightarrow{r_L} X$ is just $b$.
\end{proof}

\begin{corollary} \label{corr:main}
Let $k$ be a perfect field of exponential characteristic $e \ne 2$.
The functor $\rho: \NAlg(\HI_0(k))[1/e] \to T^2(k)[1/e]$ is an equivalence of categories.
\end{corollary}
\begin{proof}
The functors $U_2$ and $U$ are right adjoints that preserve sifted, hence filtered, colimits,
and are
conservative. See Lemmas \ref{lemm:NAlg-HI-adjoint} and \ref{lemm:T1-pres}, and
Corollary \ref{corr:T1-sifted-colimits}.
It follows that $\rho$ preserves filtered colimits and (small) limits.
Consequently $\rho$ is an accessible functor which preserve limits. It
follows that it has a left adjoint \cite[Corollary 5.5.2.9]{lurie-htt}. Denote
the left adjoint of $U_2$ by $F_2: \HI_0(k)
\to T^2(k)$, and the left adjoint of $\rho$ by $\delta: T^2(k) \to
\NAlg(\HI_0(k))$. By the Barr-Beck-Lurie Theorem\footnote{In light of Remark
\ref{rmk:1-categorical} we are dealing with 1-categories only, so we are really
just using the classical monadicity theorem \cite[Section VI.7]{mac2013categories}.}, both functors $U_2$ and $U$ are monadic \cite[Theorem
4.7.4.5]{lurie-ha}. Let $M_2 = U_2 F_2$ and $M = UF$ denote the corresponding
monads. We obtain a morphism of monads $\alpha: M_2 = U_2 F_2 \Rightarrow U_2 \rho
\delta F_2 \wequi UF = M$. If $X \in \SmQP_k$ then $F_2 \E X [1/e]
\wequi \rho F \E X [1/e]$ by Theorem \ref{thm:main}, and so $U_2 F_2 \E X[1/e]
\wequi U_2 \rho F \E X [1/e] \wequi U_2 \rho \delta F_2 \E X [1/e]$.
In other words, $\alpha(\E X[1/e])$ is an
equivalence. Since $U, U_2, F, F_2$ preserve sifted colimits, $M, M_2$ preserve
sifted colimits. Since $\SH(k)^\veff$ is generated under colimits by
$\Sigma^\infty_+ \SmQP_k$ \cite[Remark after Proposition 4]{bachmann-very-effective},
its localization $\HI_0(k)[1/e]$ is generated under
colimits by $\E(\SmQP_k)[1/e]$. Since $\E(\SmQP_k)[1/e]$ is closed under finite coproducts, it
follows from Lemma \ref{lemm:colim-closure} below that $\HI_0(k)$ is generated under
sifted colimits by $\E(\SmQP_k)[1/e]$. Since $\alpha$ preserves sifted colimits and is
an equivalence on the generators, it is an equivalence in general. This
concludes the proof.
\end{proof}

We have used the following well-known result.
\begin{lemma} \label{lemm:colim-closure}
Let $\scr C$ be an $\infty$-category with (small) colimits and $S$ a (small)
set of objects closed under finite coproducts. The subcategory of $\scr C$ generated by $S$ under
sifted colimits coincides with the subcategory of $\scr C$ generated by $S$
under all (small) colimits.
\end{lemma}
\begin{proof}
Let $\scr D \subset \scr C$ be the subcategory generated by $S$ under sifted
colimits. It suffices to show that $\scr D$ is closed under finite coproducts
\cite[Lemma 5.5.8.13]{lurie-htt}. Since $\emptyset \in S \subset \scr D$, it
suffices to consider binary coproducts. For $E \in \scr D$ write $\scr D_E \subset
\scr D$ for the subcategory of those $D \in \scr D$ with $E \coprod D \in \scr
D$. Since sifted simplicial sets are contractible, $\scr D_E$ is closed under
sifted colimits. Let $X \in S$. Then $S \subset \scr D_X$, and so $\scr D = \scr
D_X$. In other words, for $E \in \scr D$ and $X \in S$ we have $E \coprod X \in
\scr D$. Thus for $E \in \scr D$ arbitrary we have $S \subset \scr D_E$. It
follows again that $\scr D = \scr D_E$ and so $\scr D$ is closed under binary
coproducts, as desired.
\end{proof}

\begin{remark}
We have $T^1_{\Sm_k}(k) \wequi T^1_{\SmQP_k}(k)$ by Remark \ref{rmk:T1-C-comp}, and
$\NAlg_{\Sm_k}(\SH^{\eff\heart}) \wequi \NAlg_{\SmQP_k}(\SH^{\eff\heart})$ by
\cite[Remark 16.26]{bachmann-norms}. Hence in the statament of Corollary
\ref{corr:main}, we may replace $T^2(k)$ by $T^1_{\Sm_k}(k)$ and
$\NAlg(\HI_0(k))$ by $\NAlg_{\Sm_k}(\HI_0(k))$, which may be slightly more
natural choices.
\end{remark}

\begin{remark} \label{rmk:1-2}
Throughout this section we put $\scr C = \SmQP_k$ and $\scr V = \SmQP$. We
cannot reasonably hope to change $\scr V$. However, $\scr C = \SmQP_k$ was
mainly used as a simplifying assumption. It implies that $T^2_{\scr C}(k) \wequi
T^1_{\scr C}(k)$. This latter category is somewhat easier to study, and so we
were able to deduce somewhat cheaply that $T^2_\scr{C}(k)$ is presentable, and
so on. I contend that all the results about $T^2_\scr{C}$ remain valid for more
general $\scr C$, such as $\scr C = \FEt_k$ (except of course that in general
$T^2_\scr{C}(k) \not\wequi T^1_\scr{C}(k)$),
and that the same is true for Corollary \ref{corr:main}
\end{remark}

\begin{remark}
We were forced to invert the exponential characteristic $e$ essentially because
we needed to know that $(\E X)(K)$ is generated by transfers along finite
\emph{étale} morphisms, in some sense.
The proof shows that in general, $(\E
X)(K)$ is generated by transfers along finite, but not necessarily étale,
morphisms (see Remark \ref{rmk:non-perfect-gens}). Inverting $e$ allows us to restrict to $K$ perfect, where these two
classes of morphisms coincide.
\end{remark}

\bibliographystyle{plainc}
\bibliography{bibliography}

\end{document}

%% file: preamble2.tex
\usepackage{amsmath}
\usepackage{amssymb}
\usepackage{amsthm}
\usepackage{amscd}
\usepackage{enumerate}
\usepackage[pdfusetitle,unicode,hidelinks]{hyperref}
\usepackage{bbm}
\usepackage{etoolbox}

\usepackage[utf8]{inputenc}

\newcommand{\tomemail}{\href{mailto:tom.bachmann@zoho.com}{tom.bachmann@zoho.com}}

\newtheorem{proposition}{Proposition}
\newtheorem{corollary}[proposition]{Corollary}
\newtheorem{lemma}[proposition]{Lemma}
\newtheorem{theorem}[proposition]{Theorem}

\newtheorem*{theorem*}{Theorem}
\newtheorem*{corollary*}{Corollary}
\newtheorem*{proposition*}{Proposition}
\newtheorem*{lemma*}{Lemma}
\theoremstyle{definition}
\newtheorem{definition}[proposition]{Definition}
\newtheorem{construction}[proposition]{Construction}

\newtheorem*{definition*}{Definition}
\newtheorem*{construction*}{Construction}
\theoremstyle{remark}
\newtheorem{remark}[proposition]{Remark}

\newtheorem{example}[proposition]{Example}

\newcommand{\id}{\operatorname{id}}
\newcommand{\Z}{\mathbb{Z}}

\let\scr=\mathcal

\newcommand{\Gm}{{\mathbb{G}_m}}
\newcommand{\Gmp}[1]{{\mathbb{G}_m^{\wedge #1}}}

\newcommand{\1}{\mathbbm{1}}

\newcommand{\eff}{{\text{eff}}}
\newcommand{\veff}{{\text{veff}}}
\newcommand{\DM}{\mathcal{DM}}
\newcommand{\SH}{\mathcal{SH}}

\DeclareMathOperator*{\colim}{colim}

\let\lim=\relax
\DeclareMathOperator*{\lim}{lim}

\def\CAlg{\mathrm{CAlg}}
\def\NAlg{\mathrm{NAlg}}

\def\Span{\mathrm{Span}}
\def\Cat{\mathcal{C}\mathrm{at}{}}

\def\Fin{\cat F\mathrm{in}}

\def\red{\mathrm{red}}

\newcommand{\iso}{\cong}
\newcommand{\wequi}{\simeq}
\newcommand{\Mod}{\text{-}\mathcal{M}od}

\DeclareRobustCommand{\ul}{\underline}
\newcommand{\heart}{\heartsuit}

\newcommand{\Hom}{\operatorname{Hom}}

\def\op{\mathrm{op}}
\newcommand{\HI}{\mathbf{HI}}

\let\cat=\mathrm
\def\Sm{{\cat{S}\mathrm{m}}}
\def\Sch{\cat{S}\mathrm{ch}{}}